\documentclass[11pt,twoside]{article}
\usepackage{amssymb,amsbsy,amsmath,amsfonts,amssymb,amscd,amsthm}
\usepackage{epsfig, graphicx,subfigure}
\usepackage[dvips,pdfstartview = FitH,colorlinks, urlcolor=green, linkcolor=blue]{hyperref}
\usepackage{psfrag}
\usepackage{float}
\numberwithin{equation}{section}
\usepackage{colortbl}
\newcommand{\commentout}[1]{}
\newcommand{\R}{\mathbb{R}}
\newcommand{\N}{\mathbb{N}}

\newcommand{\1}{{\mathchoice {\rm 1\mskip-4mu l} {\rm 1\mskip-4mu l}
{\rm 1\mskip-4.5mu l} {\rm 1\mskip-5mu l}}}

\newcommand {\lb} {\lambda}
\newcommand {\Chi} {{\bf \raise 2pt \hbox{$\pi$}} }
\DeclareMathOperator*{\argmax}{arg\,max}
\newcommand {\f}   {\frac}
\newcommand {\p}   {\partial}

\newcommand{\dis}{\displaystyle}
\newcommand{\beq}{\begin{equation}}
\newcommand{\beqa} {\begin{array}{rl}}
\newcommand{\eeq}{\end{equation}}
\newcommand{\eeqa}{\end{array}}
\newtheorem{theorem}{Theorem}
\newtheorem{lemma}[theorem]{Lemma}

\newtheorem{remark}[theorem]{Remark}
\newtheorem{proposition}[theorem]{Proposition}
\newtheorem{corollary}[theorem]{Corollary}

\title{\Large \bf Optimization of an Amplification Protocol for Misfolded Proteins by using Relaxed Control}

\author{Jean-Michel Coron\thanks{Sorbonne Universit\'es, UPMC Univ Paris 06, UMR 7598, Laboratoire Jacques-Louis Lions, F-75005, Paris, France . E-mail:
\texttt{coron@ann.jussieu.fr}. JMC was supported by ERC advanced
grant 266907 (CPDENL) of the 7th Research Framework Programme
(FP7).} \and Pierre Gabriel\thanks{Universit\'{e} de Versailles Saint-Quentin-en-Yvelines, Laboratoire de Math\'{e}matiques de Versailles, CNRS UMR 8100, 45 Avenue des \'{E}tats-Unis, 78035 Versailles, France. E-mail:
\texttt{pierre.gabriel@uvsq.fr}. PG was partially supported by
ERC advanced grant 266907 (CPDENL) of the 7th Research Framework
Programme (FP7).} \and Peipei Shang\thanks{Department of Mathematics, Tongji
University, Shanghai 200092, China. E-mail:
\texttt{peipeishang@hotmail.com}. PS was partially supported by
ERC advanced grant 266907 (CPDENL) of the 7th Research Framework
Programme (FP7) and by the National Natural Science Foundation of China (No. 11301387 ).}}

\vspace{2mm}

\date{}

\begin{document}
\maketitle
\pagestyle{plain}
\pagenumbering{arabic}

\begin{abstract}
We investigate an optimal control problem which arises in the optimization of an amplification technique for misfolded proteins. The improvement of this technique may play a role in the detection of prion diseases.
The model consists in a linear system of differential equations with a nonlinear control.
The appearance of oscillations in the numerical simulations is understood by using the Perron and Floquet eigenvalue theory for nonnegative irreducible matrices.
Then to overcome the unsolvability of the optimal control, we relax the problem.
In the two dimensional case, we solve explicitly the optimal relaxed control problem when the final time is large enough.
\end{abstract}


\


\


\section{Introduction}

Transmissible Spongiform Encephalopathies (TSE) are fatal, infectious, neurodegenerative diseases.
They include bovine spongiform encephalopathies (BSE) in cattle, scrapie in sheep and Creutzfeldt-Jakob disease (CJD) in human.
During the so-called ``mad-cow crisis'' in the 90's, people were infected by a variant of BSE by ingesting contaminated pieces of beef.
More recently, CJD was transmitted between humans via blood or growth hormones.
Because of the long incubation times (some decades), TSE still represent an important public health risk.
There is no \emph{ante mortem} diagnosis currently available to detect infected individuals and prevent possible contaminations.
A promising tool to design a diagnosis test is the protein misfolded cyclic amplification (PMCA) technique~\cite{Castilla,Saa,Saborio}.

The PMCA principle is based on the ``protein-only hypothesis''~\cite{Griffith,Prusiner}.
According to this widely accepted hypothesis, the infectious agent of TSE, known as prions, may consist in misfolded proteins called PrPsc (for Prion Protein scrapie).
The PrPsc replicates in a self-propagating process, by converting the normal form of PrP (called PrPc for Prion Protein cellular) into PrPsc.
The PMCA enabled to consolidate the idea of an autocatalytic replication of PrPsc by nucleated polymerization.
In this model originally proposed by Landsbury~\cite{Lansbury}, PrPsc is considered to be a polymeric form of PrPc.
Polymers can lengthen by addition of PrPc monomers, and they can replicate by splitting into smaller fragments.
The PrPc is mainly expressed by the cells of the central nervous system, so PrPsc concentrates in this zone.
The amount of PrPsc in tissues such as blood is very small and this is why it is very difficult to diagnose an infected individual.

The PMCA mimics {\it in vitro} the nucleation/polymerization mechanism occurring \emph{in vivo} with the aim to quickly amplify the polymers present in minute amount in an infected sample.
It is a cyclic process, where each cycle consists in two phases:
the {\it incubation} phase during which the polymerization is favored due to the presence of a large quantity of PrPc monomers,
and the {\it sonication} phase when the PrPsc polymers are broken down with ultrasounds.
The {\it incubation} phase is expected to increase the size of the polymers,
while the ultrasounds are known to increase the fragmentation of the polymers and so increase their number.
This technique could allow us to detect PrPsc in the samples of blood for instance.
But for now, it is not efficient enough to do so.
Mathematical modelling and optimization tools can help to optimize the PMCA protocol.

\

The mathematical modeling of prion proliferation with ordinary or partial differential equation (PDE) produced a large literature since the first model of Griffith~\cite{Griffith}.
Today, the most widely studied nucleation/polymerization model is the one of Masel~\cite{Masel}.
A PDE version of this model has been introduced by Greer {\it et al.}~\cite{Greer} and studied by many authors including~\cite{CL2,CL1,DGL,GabrielPhD,PG,LW,Pruss,SimonettWalker,Walker}.
Starting from it, we propose to model the PMCA with the following controlled size-structured PDE
\beq\label{eq:prionPDE}\p_tf(t,\xi)+r(u(t))\p_\xi\bigl(\tau(\xi)f(t,\xi)\bigr)=u(t)\left(\int_\xi^\infty\beta(\zeta)\kappa(\xi,\zeta)f(t,\zeta)\,d\zeta-\beta(\xi)f(t,\xi)\right)\eeq
with the boundary condition $f(t,0)=0$ for every time $t\geq0$.
The unknown $f(t,\xi)$ is the number, or density of polymers of size $\xi>0$ at time $t$.
The size of the polymers increases by polymerization with respect to the individual growth rate $\tau(\xi)$.
The terms in the large brackets on the right-hand side of \eqref{eq:prionPDE} form the fragmentation operator, with $\beta(\xi)$ the global fragmentation rate and $\kappa(\xi,\zeta)$ the fragmentation kernel.
The conservation of the quantity of polymerized proteins during the fragmentation process requires that this kernel satisfies the following standard condition (see~\cite{DG} for instance)
\begin{gather}
\label{massconservation}
\int_0^\zeta \xi\,\kappa(\xi,\zeta)\,d\xi=\zeta.
\end{gather}
The fragmentation is modulated by a multiplicative factor $u(t)\in[u_{\min},u_{\max}]$ which represents the {\it sonication} intensity.
The control $u(t)\equiv u_{\min}=1$ corresponds to the absence of {\it sonication}, while $u(t)\equiv u_{\max}>1$ represents the maximal power of the {\it sonicator}.
We assume that the {\it sonication} does not only increase the fragmentation but also influence the polymerization process.
This is taken into account by the positive term $r(u(t))$ where the function $r$ should be decreasing if we consider that the ultrasounds have a negative effect on the growth of the polymers.
The optimal control problem we are interested in is, starting with a given initial size distribution $f(t=0,\xi)=f^0(\xi)\geq0$,
to maximize the objective
\begin{equation}\label{costJ}
J(u)=\int_0^\infty \xi f(T,\xi)\,d\xi,
\end{equation}
which represents the total quantity of polymerized proteins at a given final time $T$.

\

For the mathematical study in this paper, we consider a $n$-compartment approximation of \eqref{eq:prionPDE}
\[\dfrac{dx_i}{dt} + r(u(t))(\tau_i x_i - \tau_{i-1} x_{i-1}) = u(t)\left(\sum_{j=i+1}^n \beta_j\, \kappa_{i,j}\,x_j -\beta_i x_i\right) \ ,
\quad \mbox{for} \ 1\leq i\leq n,\]
with $\beta_1=0$ and $\tau_n=0$. This is a finite dimensional ordinary differential system, linear in $x=(x_1,\dots,x_n)^T$, which can be written under a matrix form
\beq\label{eq:matrixformu}
\left \{ \begin{array}{l}
\dot{x} = uFx+r(u)Gx, \qquad t\in[0,T],
\vspace{2mm}\\
x(t=0) = x^{0}>0,
\end{array} \right.\eeq
where $G$ is the growth matrix
\begin{equation}
\label{defG}
G = \left(\begin{array}{ccccc}
-\tau_{\text{\scriptsize $1$}}& & & & \\
\tau_{\text{\scriptsize $1$}} &-\tau_{\text{\scriptsize $2$}}& & \text{\LARGE$0$} & \\
 &\ddots&\ddots& & \\
 &  & & & \\
   & \text{\LARGE$0$} &\tau_{\text{\scriptsize $n-2$}}&-\tau_{\text{\scriptsize $n-1$}}& \\
 & & &\tau_{\text{\scriptsize $n-1$}}&0
\end{array}\right) ,
\end{equation}
and $F$ is the fragmentation matrix
\begin{gather}
\label{defF}
F = \left(\begin{array}{cccc}
0& & & \\
 &-\beta_2& &\text{{\Large$(\kappa_{ij}\beta_j)_{i<j}$}}\\
 & &\ddots& \\
 &\text{\LARGE$0$}& & \\
 & & &-\beta_n
\end{array}\right) .
\end{gather}
In \eqref{eq:matrixformu} and in the following, if $x=(x_1,\dots,x_n)^T$, by $x>0$ (and we also write $x$ is positive) we mean that $x_i>0$ for every $i\in \{1,\cdots,n\}.$ We use the same notation for row vectors.
\

We assume that
\begin{equation}
\label{constpositive}
\tau_i>0\quad\text{and}\quad\beta_{i+1}>0, \qquad \forall i\in[1,n-1].
\end{equation}
The mass conservation assumption \eqref{massconservation} on $\kappa$ becomes
\beq\label{as:kappadiscret}
\sum_{i=1}^{j-1}i\,\kappa_{ij}=j,\qquad j=2,\cdots, n.\eeq
The quantity \eqref{costJ} we want to maximize writes
\begin{equation}\label{new-cost}
J(u)=\sum_{i=1}^n i\,x_i(T).
\end{equation}

\

Such $n$-compartment optimal control problems have been widely studied in cancer chemotherapy and the conclusion is usually that the optimal control is bang-bang since singular controls are not optimal~\cite{LedSch02b,LedSch02,LedSch06,LedSch06b,SwiLedSch}.
 In contrast with these results, we show that, for our problem, the optimal control is essentially singular.

\

The organization of the paper is the following.
In Section~\ref{sec:eigenvalue}, we investigate eigenvalue optimization problems related to the optimal control problem \eqref{new-cost}.
More precisely we maximize the Perron eigenvalue with respect to constant control parameters and compare the maximum to Floquet eigenvalues, which are the analogue of Perron eigenvalues for periodic coefficients.
We remark that fast oscillating controls can provide a greater Floquet eigenvalue than the optimal Perron eigenvalue.
This observation indicates that a classical optimal control may not exist and motivates the relaxation of the problem.
The relaxed control problem, for which the set of controls is extended to its convex hull, is investigated in Section \ref{sec:relax}.
We prove that the trajectories corresponding to the constant control which maximizes the Perron eigenvalue in the new convex set satisfy the Pontryagin Maximum Principle (see Proposition~\ref{prop:exsolPerron}).
In Section \ref{secn=2}, we state and prove the main result of the paper (Theorem~\ref{th:uniqueness}) which treats the case of the two-compartment model:
for $n=2$ the optimal relaxed control is unique and can be computed explicitly.
Except for initial and terminal times, {\it i.e.}, $t$ close to $0$ or close to $T$, this optimal control is equal to the constant which maximizes the Perron eigenvalue.
Finally, in the Appendix, we give the details of the proofs for the results of Section~\ref{sec:eigenvalue}.

\

\section{Eigenvalue problems}\label{sec:eigenvalue}

For a fixed parameter $u>0$ and for $r(u)>0,$ the matrix $uF+r(u)G$ is irreducible (see, for instance, \cite[Section 2.8]{2002-Serre-book} for a definition of irreducible) and has nonnegative extra-diagonal entries. So the Perron-Frobenius theorem (see, for instance,
\cite[Section 5.3]{2002-Serre-book}) applies and ensures the existence of a simple dominant eigenvalue $\lb_P$.
In our case, this eigenvalue is positive and it provides the exponential growth rate of the solutions to the equation $\dot x=(uF+r(u)G)x$ (see, for instance, \cite[Section 6.3.1]{BP}).
A first question is to investigate the dependence of the first eigenvalue on the parameter $u$.
Maximizing the Perron eigenvalue is related to our optimal control problem \eqref{new-cost}.
It can be regarded as the limit when $T\to+\infty$ of our optimization problem when we restrict to constant controls.
A remarkable fact is that for some coefficients, the dependence $u\mapsto\lambda_P(u)$ can be non monotonic and there may exist an optimal value $u_\mathrm{opt}$ for
which $\lb_P$ admits a global maximum on $\R^+$.
Theorem~\ref{th:Perron}, which is proved in Appendix \ref{sec-Perron-eigenvalue},
gives sufficient conditions for the existence of such a global optimum.

\begin{theorem}\label{th:Perron}
Assume that $r:\R^+\to\R^{+*}$ is continuous and  admits an expansion of the form
\beq\label{as:r}\exists\, l>0,\, r_l\geq0,\qquad r(u)=r_0+r_lu^{-l}+\underset{u\to +\infty}{o}\left(u^{-l}\right).\eeq
Consider also that $(\tau_i)_{1\leq i\leq n}$ satisfies the condition
\beq\label{as:tau} \exists\, k\in\N^*\quad \text{such that}\quad\forall\, i\leq k,\ \tau_i=i\,\tau_1\quad\text{and}\quad \tau_{k+1}>(k+1)\tau_1.\eeq
Then there exists an optimal value $u_\mathrm{opt}>0$ which satisfies
$$\forall\,u\geq0,\qquad \lambda_P(u)\leq\lambda_P(u_\mathrm{opt}).$$
\end{theorem}

The interpretation is that in this case, there is a compromise between too much sonication which forms many small polymers but may have a small growth rate, and too high sonication which forms large polymers but in small quantity.

\

The theory of Perron-Frobenius can be extended to periodic controls: this is the Floquet theory. It ensures that for time periodic matrices which are monotonic and irreducible for any time, there exists a dominant eigenvalue.
It allows one to define for a periodic control $u(t)$ a Floquet eigenvalue $\lb_F[u]>0$ which prescribes, as in the case of the Perron eigenvalue, the asymptotic exponential growth rate of the solutions to the equation $\dot x=(u(t)F+r(u(t))G)x$
(see~\cite[Section 6.3.2]{BP} for instance).
A natural question is then to compare these periodic eigenvalues to the best constant one $\lb_P(u_\mathrm{opt})$.
Theorem~\ref{th:Floquet}, which is proved in Appendix \ref{secFloquet},  ensures that if $r$ satisfies the condition
\beq\label{as:r2}\f{r''(u_\mathrm{opt})}{r(u_\mathrm{opt})-u_\mathrm{opt}\,r'(u_\mathrm{opt})}>0,\eeq
then the value $u_\mathrm{opt}$ is a saddle point in the set of periodic controls.
This means that there exist periodic controls which provide a larger growth rate than $\lb_P(u_\mathrm{opt})$.

\begin{theorem}\label{th:Floquet}
Assume that there exists an optimal value $u_\mathrm{opt}$ for the Perron eigenvalue and that $u_\mathrm{opt} F+r(u_\mathrm{opt})G$ is diagonalizable.
Define for a frequency $\omega>0$ and a perturbation $\varepsilon>0$ the Floquet eigenvalue $\lb_F(\varepsilon,\omega):=\lb_F[u_\mathrm{opt}+\varepsilon\cos(\omega t)]$.
Then we have
\[\lim_{\omega\to+\infty}\f{\p^2}{\p\varepsilon^2}\lb_F(0,\omega)=\f12\,\f{r''(u_\mathrm{opt})}{r(u_\mathrm{opt})-u_\mathrm{opt}\,r'(u_\mathrm{opt})}\,\lb_P(u_\mathrm{opt}).\]
\end{theorem}

\

The computation of second order approximation of the Floquet eigenvalue is used to detect ``resonance'' in population models with periodic coefficients (see~\cite{Bacaer2} and the references therein).
For these models, resonance is said to occur if periodic variations in the coefficients increase the growth rate.

\medskip

The link between the eigenvalue problem and the optimal control problem \eqref{new-cost} is investigated in~\cite{CalvezGabriel} when the function $r$ is a constant.
In this case, there exists an optimal control $u^*(t)$ which is essentially equal to the best constant $u_\mathrm{opt}$ of the Perron optimization problem
(see Chapter~5 in~\cite{GabrielPhD} for numerical simulations).
Under condition~\eqref{as:r2} and the assumptions of Theorem~\ref{th:Floquet}, such a behavior is not expected since we can find oscillating controls which provide a better eigenvalue than $u_\mathrm{opt}$.
The aim of this paper is to investigate the optimal control problem \eqref{new-cost} in the case where there exists an optimal constant $u_\mathrm{opt}\in(u_{\min},u_{\max})$ and the function $r$ satisfies Assumption~\eqref{as:r2}.
The first question is the existence of an optimal control since the numerical simulations in Figure~\ref{fig:r_convex} show oscillations.
These questions are investigated in the following section by using the relaxed control theory.

\

\begin{figure}[ht!]
\begin{center}
\begin{minipage}[b]{0.49\linewidth}
\includegraphics [width=\linewidth]{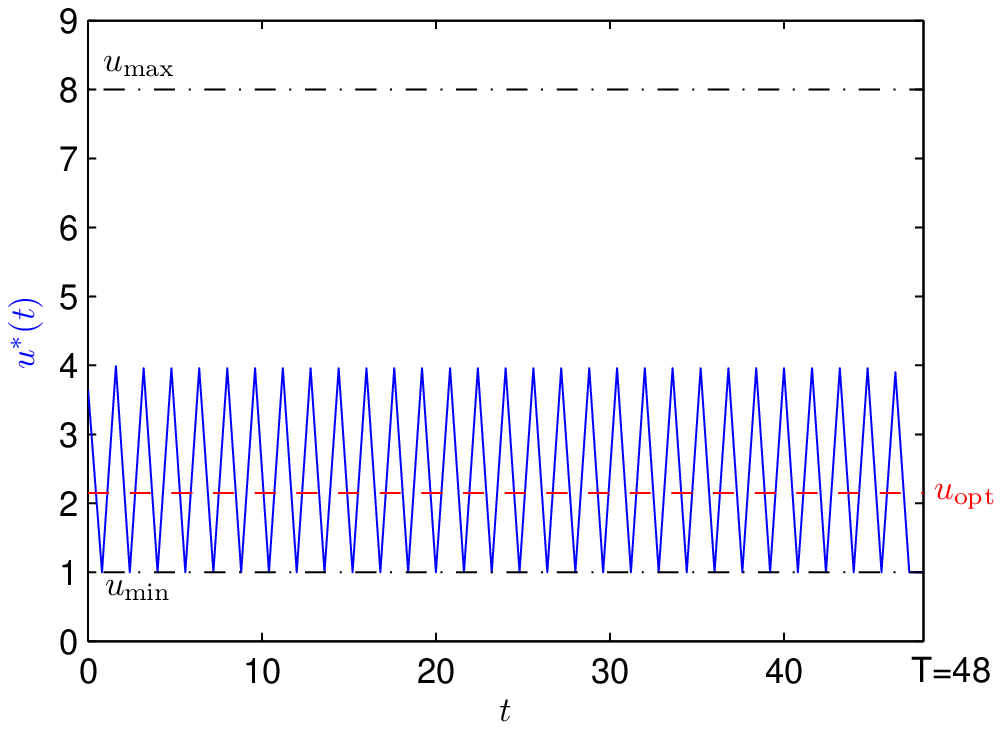}
\end{minipage}\hfill
\begin{minipage}[b]{0.49\linewidth}
\includegraphics [width=\linewidth]{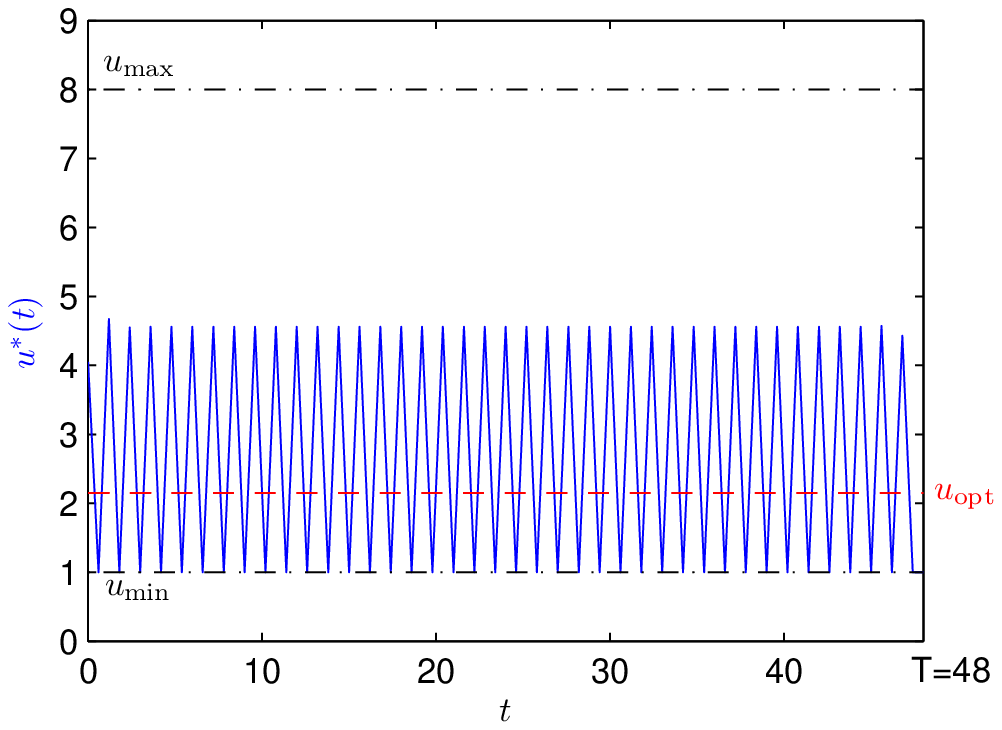}
\end{minipage}
\end{center}

\begin{center}
\begin{minipage}[b]{0.49\linewidth}
\includegraphics [width=\linewidth]{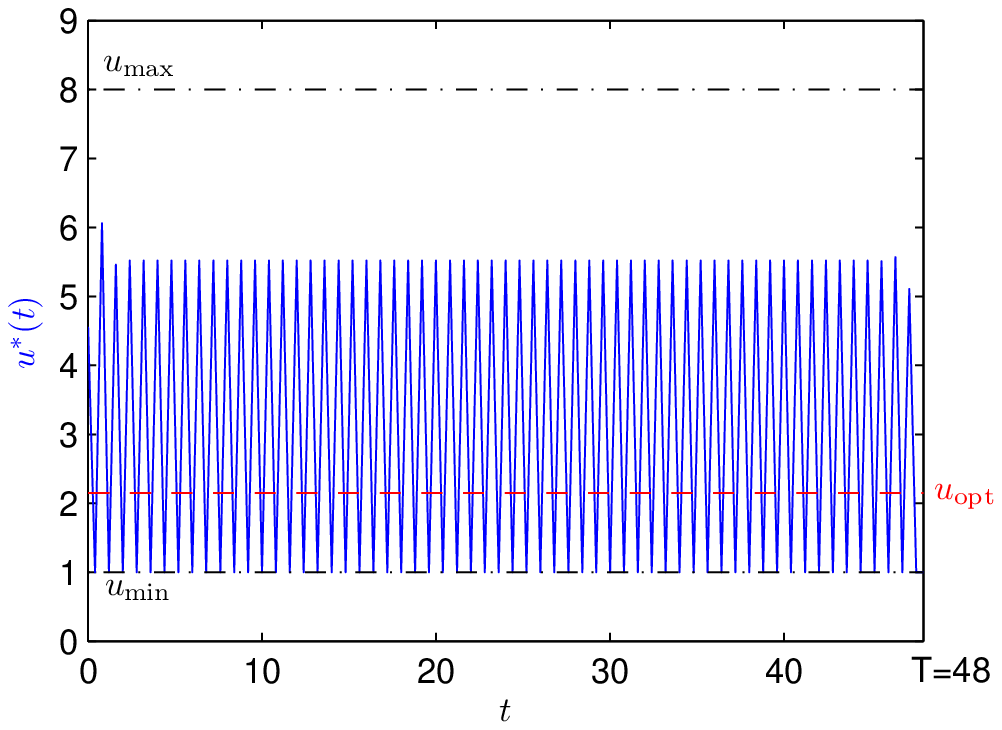}
\end{minipage}\hfill
\begin{minipage}[b]{0.49\linewidth}
\includegraphics [width=\linewidth]{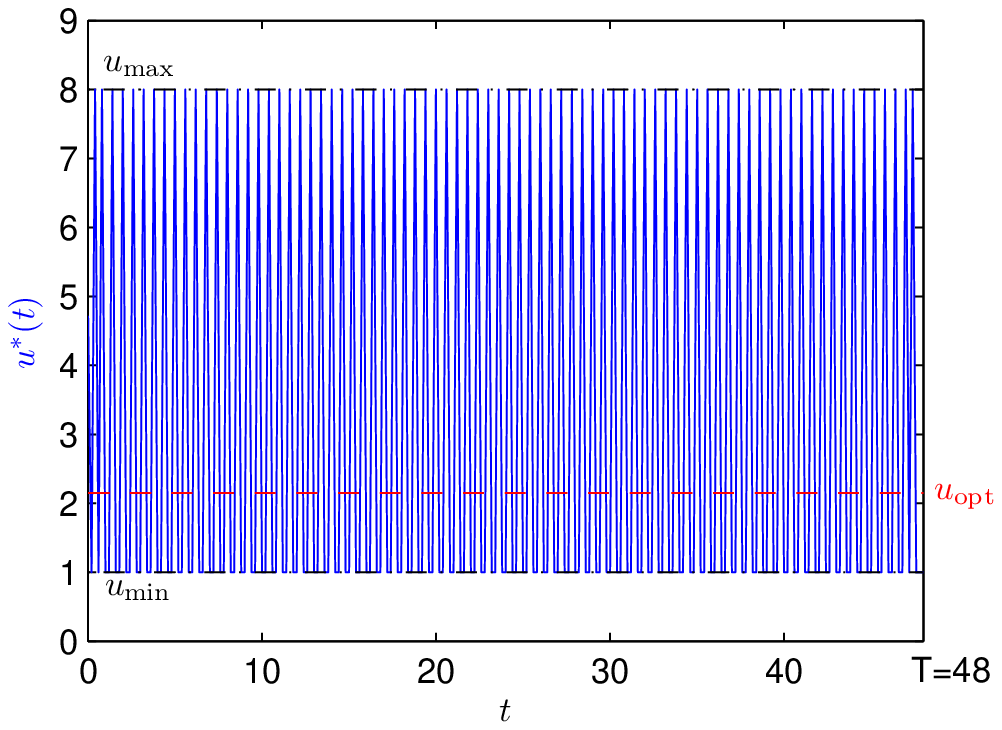}
\end{minipage}
\end{center}
\caption{\label{fig:r_convex}Piecewise optimal controls $u^*(t)$ obtained for different values of $\Delta t$ with the decreasing convex function $r(u)=\f{2}{1+u}.$
The control varies between $u_{\min}=1$ and $u_{\max}=8$ and the final time is $T=48.$
The dimension of the model is $n=3$ and the coefficients are $\tau_1=1,$ $\tau_2=10,$ $\beta_2=0.5,$ $\beta_3=1,$ $\kappa_{1,2}=2$ and $\kappa_{1,3}=\kappa_{2,3}=1.$
The time step varies as follows: $\Delta t=0.8$ (top left), $\Delta t=0.6$ (top right), $\Delta t=0.4$ (bottom left) and $\Delta t=0.2$ (bottom right).}
\end{figure}

\

\section{Relaxed Control and the Pontryagin maximum principle}\label{sec:relax}

Let $u_{\min}$ and $u_{\max}$ be two positive real numbers such that $u_{\min}<u_{\max}$.
We consider in this section a function $r\in\mathcal C^2([u_{\min},u_{\max}])$ satisfying
\begin{gather}
\label{rpositive}
r(u)>0,\qquad \forall u \in [u_{\min},u_{\max}],
\\
\label{rsecondpositive}
r''(u)>0, \qquad \forall u\in (u_{\min},u_{\max}),
\\
\label{as:r3}
r(u)-u r'(u)>0, \qquad \forall u\in (u_{\min},u_{\max}),
\end{gather}
and we assume that there exists a strict optimum $u_\mathrm{opt}\in(u_{\min},u_{\max})$ for the Perron eigenvalue.
Under these assumptions, condition~\eqref{as:r2} is automatically fulfilled.
Remark that condition~\eqref{as:r3} is satisfied when $r$ is decreasing and satisfies \eqref{rpositive}, which are relevant conditions from the viewpoint of biology.
To study this case, it will be convenient to use the equivalent alternative statement of \eqref{eq:matrixformu} where $x$ is solution to
\beq\label{eq:matrixformuv}
\left \{ \begin{array}{l}
\dot{x}(t) = u(t)F x(t)+v(t)G x(t), \qquad t\in[0,T],
\vspace{2mm}\\
x(t=0) = x^{0}>0,
\end{array} \right.\eeq
with the two dimensional control $(u,v)$ which belongs to the graph of the function $r$, \textit{i.e.,}
\[\forall\,t\in[0,T],\qquad(u(t),v(t))\in\Omega:=Graph(r)=\left\{(u,r(u)),\ u_{\min}\leq u\leq u_{\max}\right\}.\]
Let
\begin{gather}
\label{defpsi}
\psi:=(1,2,\cdots, n) \in \R^n
\end{gather}
be the \emph{mass vector}. Note that, from \eqref{defF}, the mass conservation assumption~\eqref{as:kappadiscret} and \eqref{defpsi}, one has
\begin{gather}
\label{psiF}
\psi F=0.
\end{gather}

The optimal control problem \eqref{new-cost} now becomes
\begin{gather}
\label{eq:cost1}
\text{maximize } J(u,v)=\psi x(T),\, (u,v):[0,T]\rightarrow \Omega  \text{ is a Lebesgue measurable function},
\end{gather}
subject to dynamics~\eqref{eq:matrixformuv}.
Since the function $r$ is strictly convex, the graph $\Omega$ is not a convex subset of $\R^2$ and, since the kernel of $G$ is not reduced to $\{0\}$, for any $x>0$, the velocity set
\[{\mathcal V}(x)=\{(u F+vG)x,\ (u,v)\in\Omega\}\]
is also not convex. For this kind of situation, the existence of an optimal control for \eqref{eq:cost1} cannot be ensured
and it is standard to relax the problem by replacing the control set $\Omega$ by its convex hull $Conv(\Omega)$ (see, for instance, \cite[Section 4.2]{LeeMarkus}).
One replaces problem \eqref{eq:cost1} by the following optimal control problem
\beq
\label{eq:cost2}
\text{maximize } J(u,v)=\psi x(T),\, (u,v):[0,T]\rightarrow Conv(\Omega)  \text{ is a Lebesgue measurable function},
\eeq
subject to dynamics~\eqref{eq:matrixformuv}.
For this problem, the velocity set is the convex hull of $\mathcal V(x)$, so it is convex and the existence of an optimal control is ensured by classical results
(see~\cite[Theorem 5 p.~271]{LeeMarkus} for instance). Moreover,
\begin{itemize}
\item The supremum in problem \eqref{eq:cost1} is equal to the maximum in problem \eqref{eq:cost2}.
\item Let $(u^*,v^*): [0,T]\rightarrow Conv(\Omega)$ be a measurable function which is optimal for problem \eqref{eq:cost2}. Then one can easily construct a sequence of piecewise constant
functions $(u^n,v^n)_{n\in \mathbb{N}}\rightarrow  \Omega$  such that, for any $\varphi\in L^{\infty}(0,T)$,
\begin{gather}
\label{weakconvergence}
\int_0^T u^n\varphi\, dt \longrightarrow  \int_0^T u^*\varphi\, dt \quad \text{ and }
\int_0^T v^n\varphi\, dt\longrightarrow \int_0^T v^*\varphi\, dt\quad \text{as}\quad n\rightarrow +\infty.
\end{gather}
Let us emphasize that \eqref{weakconvergence} implies that
$$
J(u^n,v^n)\rightarrow J(u^*,v^*) \text{ as } n\rightarrow +\infty.
$$
In particular, $(u^n,v^n)_{n\in \mathbb{N}}$ is a maximizing sequence for problem \eqref{eq:cost1}.
\end{itemize}

Now we want to obtain information on the optimal controls for \eqref{eq:cost2}
by using the Pontryagin maximum principle. This principle in our case gives the following theorem.

\begin{theorem}[Pontryagin Maximum Principle (PMP)]\label{th:PMP}
Let $(u^*,v^*)$ be an optimal control for problem \eqref{eq:cost2} and let $x^*$  be the corresponding trajectory (i.e., the solution of \eqref{eq:matrixformuv} with $(u,v):=(u^*,v^*)$).
Call $p^*:\, [0,T]\rightarrow \R^n$ the row vector solution of the adjoint linear equation
\begin{equation}\label{eq:adjoint}
\dot p^*(t)=-p^*(t)(u^*F+v^*G),
\end{equation}
with the transversality condition
\begin{equation}\label{eq:trans}
p^*(T)=\psi.
\end{equation}
Let us define the Hamiltonian as
\begin{equation}
\label{defHamiltonien}
H(x,p,u,v):= p(uF+vG)x.
\end{equation}
Then the maximality condition
\begin{equation} \label{eq:maximality}
H(x^*(t),p^*(t),u^*(t),v^*(t))=\max_{(u,v)\in Conv(\Omega)}H(x^*(t),p^*(t),u,v)
\end{equation}
holds for almost every time $t\in[0,T]$ and there exists a constant $H^*\in \R$  such that
\begin{gather}
\label{Hconstant}
H(x^*(t),p^*(t),u^*(t),v^*(t))=H^*, \, \text{ for almost every } t \in [0,T].
\end{gather}
\end{theorem}

\begin{remark}
Since, for any positive $u$ and $v$, the matrix $uF+vG$ has nonnegative extra-diagonal entries, we have, using \eqref{eq:matrixformuv}, \eqref{defpsi}, \eqref{eq:adjoint} and \eqref{eq:trans},
\begin{equation}\label{xandp>0}
x^*(t)>0 \quad \text{and}\quad p^*(t)>0,\quad \text{for every time}\quad t\in[0,T].
\end{equation}

\end{remark}

\

The Pontryagin maximum principle is useful to obtain information on the optimal control. It allows us to prove
(Corollary~\ref{corollarySigma}~below) that the optimal control lies on the line $\Sigma$ defined by (see Figure~\ref{fig:convexhull})
\beq\label{def:Sigma}\Sigma:=Graph(\sigma)=\left\{(u,\sigma(u)),\ u_{\min}\leq u\leq u_{\max}\right\},\eeq
where $\sigma$ is the affine function defined by
\[\sigma(u)=\theta u+\zeta\]
with
\begin{gather}
\label{deftheta}
\theta:=\f{r(u_{\max})-r(u_{\min})}{u_{\max}-u_{\min}},
\\
\label{defzeta}
\zeta:=\f{u_{\max}r(u_{\min})-u_{\min}r(u_{\max})}{u_{\max}-u_{\min}}.
\end{gather}
The set $\Sigma$ is the string which links $(u_{\min},r(u_{\min}))$ to $(u_{\max},r(u_{\max}))$.
Since $r$ is convex, the boundary of the control set $Conv(\Omega)$ is $\partial Conv(\Omega)=\Omega\cup\Sigma$.

One has the following lemma which is illustrated by the incoming arrows in Figure~\ref{fig:convexhull}.
\begin{lemma}\label{lm:convex}
Let $(u,v)\in Conv(\Omega)\setminus\Sigma$. Then, for $\varepsilon > 0$ small enough,
\[((1+\varepsilon)u,(1+\varepsilon)v)\in Conv(\Omega).\]
\end{lemma}

\begin{figure}[h!]
\psfrag{Omega}[l]{$\Omega$}
\psfrag{H(Omega)}[l]{$Conv(\Omega)$}
\psfrag{Sigma}[l]{$\Sigma$}
\psfrag{u}[l]{$u$}
\psfrag{v}[l]{$v$}
\psfrag{ru}[l]{$r(u)$}
\begin{center}
\includegraphics [width=.55\linewidth]{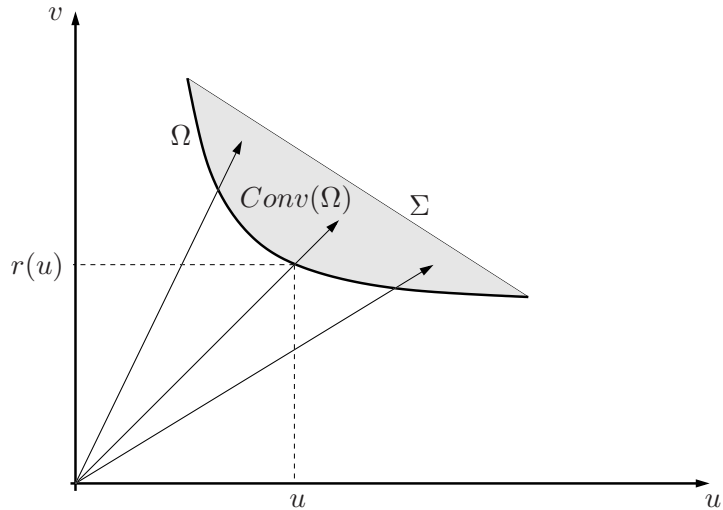}
\end{center}
\caption{\label{fig:convexhull}The set $\Omega,$ the string $\Sigma$ and the convex hull $Conv(\Omega)$ for $r$ a decreasing and convex function. The arrows oriented along the vectors $(u,r(u))$ point inside $Conv(\Omega)$ on the lower boundary $\Omega.$}
\end{figure}

\begin{proof}[Proof of Lemma \ref{lm:convex}]
Let $Int(Conv(\Omega))$ be the interior of $Conv(\Omega)$. If $(u,v)\in Int(Conv(\Omega))$, then the result follows from the fact that $Int(Conv(\Omega))$ is an open set by definition.
It remains to study the case where $(u,v)\in\Omega\setminus\{(u_{\min},r(u_{\min}))\}\cup\{(u_{\max},r(u_{\max}))\}$, \textit{i.e.,} $v=r(u)$ with $u\in(u_{\min},u_{\max})$.
We have, using \eqref{as:r3},
\begin{align*}
(1+\varepsilon)r(u)-r((1+\varepsilon)u)&=(1+\varepsilon)r(u)-r(u)-\varepsilon ur'(u)+\underset{\varepsilon \to 0^+}{o}(\varepsilon)\\
&=\varepsilon (r(u)-u\,r'(u)+\underset{\varepsilon \to 0^+}{o}(1))\geq 0,\quad \text{for $\varepsilon$ small and positive}.
\end{align*}
Hence, if $\varepsilon> 0$ is small enough, $r((1+\varepsilon)u)\leq (1+\varepsilon)r(u).$
Moreover, by \eqref{rsecondpositive}, we have $r(u)<\sigma(u)$ and, therefore, if $\varepsilon$ is small enough,  $(1+\varepsilon)r(u)\leq \sigma((1+\varepsilon)u)$.

The proof is complete.
\end{proof}

As a consequence of Theorem~\ref{th:PMP} and Lemma~\ref{lm:convex}, we have the following corollary.

\begin{corollary}
\label{corollarySigma}
Let $(u^*,v^*)$ be an optimal control for problem \eqref{eq:cost2}. Then, for almost every time $t\in [0,T]$, the optimal control $(u^*(t),v^*(t))\in\Sigma$.
\end{corollary}

\begin{proof}[Proof of Corollary \ref{corollarySigma}]

By \eqref{Hconstant}, there exists a sequence $(t_n)_{n\in \mathbb{N}}$ of elements in $[0,T]$ such that
\begin{gather}
\label{Hn=H*}
H(x^*(t_n),p^*(t_n),u^*(t_n),v^*(t_n))=H^*,
\\
\label{tntendtoT}
t_n\rightarrow T \quad\text{ as }\quad n\rightarrow +\infty.
\end{gather}
Extracting a subsequence if necessary we may assume, without loss of generality, that there exists $\widetilde v\in \R$ such that
\begin{gather}\label{propertyv*}
\widetilde v\in [r(u_{\max}),r(u_{\min})]\subset(0,+\infty),
\\
\label{converrgencevtn}
v^*(t_n)\rightarrow \widetilde v \quad\text{ as }\quad n\rightarrow +\infty.
\end{gather}
Letting $n\rightarrow +\infty $ in \eqref{Hn=H*}, and using \eqref{psiF}, \eqref{eq:trans}, \eqref{defHamiltonien}, \eqref{tntendtoT} and \eqref{converrgencevtn}, one gets that
\begin{equation}\label{atT}
H^*= \widetilde v \psi  G x(T).
\end{equation}
 From \eqref{defG}, \eqref{constpositive} and \eqref{defpsi}, one gets that
\begin{equation}\label{psiG>0}
\psi  G >0,
\end{equation}
which, together with \eqref{xandp>0}, implies  that
\begin{equation}\label{psiGx(T)>0}
\psi  G x(T)>0.
\end{equation}
 From \eqref{propertyv*}, \eqref{atT} and \eqref{psiGx(T)>0}, one obtains
\[H^*>0.\]
Let $t\in[0,T]$ be such that
\begin{gather}
\label{inHOmega}
(u^*(t),v^*(t))\in Conv(\Omega)\setminus\Sigma,
\\
\label{tbonpoint}
H(x^*(t),p^*(t),u^*(t),v^*(t))=H^*.
\end{gather}
 From Lemma~\ref{lm:convex} and \eqref{inHOmega}, there exists $\varepsilon>0$ such that $((1+\varepsilon)u^*(t),(1+\varepsilon)v^*(t))\in Conv(\Omega)$.
 Using \eqref{tbonpoint}, one has
\begin{eqnarray*}
\begin{array}{rcl}
H(x^*(t),p^*(t),(1+\varepsilon) u^*(t),(1+\varepsilon)v^*(t))&=&(1+\varepsilon) H(x^*(t),p^*(t),u^*(t),v^*(t))
\\
&=&
(1+\varepsilon)H^*
\\
&>&H^*=
H(x^*(t),p^*(t),u^*(t),v^*(t)),
\end{array}
\end{eqnarray*}
which shows that \eqref{eq:maximality} does not hold. Since, by Theorem \ref{th:PMP},  \eqref{eq:maximality} holds for almost every $t\in [0,T]$, this, together with \eqref{Hconstant}, concludes the proof of  Corollary \ref{corollarySigma}.
\end{proof}

\

Now we look for controls and corresponding trajectories which satisfy the optimality condition~\eqref{eq:maximality}.
To that end, we take advantage of our analysis of the Perron eigenvalue problem.
For $(u,v)\in Conv(\Omega)$, define the Perron eigenvalue $\lb_P=\lb_P(u,v)$ of the matrix $uF+vG$ and the corresponding right and left eigenvectors $X>0$ and $\phi>0$ normalized as follows
\begin{gather}
\label{defX}
\lb_P X=(u F+v G)X,\quad\|X\|_1=1,
\\
\label{defphi}
\lb_P \phi=\phi(u F+v G),\quad\phi X=1,
\end{gather}
where, for $x=(x_1,\cdots,x_n)^T $, $\|x\|_1:=\sum_{i=1}^n |x_i|$. The function $(u,v)\mapsto\lb_P(u,v)$ admits an optimum $(\bar u,\bar v)$ on the compact set $Conv(\Omega)$ (See Figure~\ref{fig:eigenvalue} for the numerical simulations). We denote by $\bar\lb_P$, $\bar X$ and $\bar\phi$ the corresponding optimal eigenelements.
First, we notice that Lemma~\ref{lm:convex} implies that, as for the optimal control, the optimum $(\bar u,\bar v)$ of $\lb_P$ belongs to $\Sigma$.

\begin{corollary}
 \label{cor:positionbarubarv} The optimal point $(\bar u,\bar v)=\argmax_{(u,v)\in Conv(\Omega)}\lb_P(u,v)$ satisfies
\[(\bar u,\bar v)\in\Sigma\setminus\bigl\{(u_{\min},r(u_{\min})),(u_{\max},r(u_{\max}))\bigr\}.\]
\end{corollary}

\begin{proof}[Proof of Corollary~\ref{cor:positionbarubarv}]
Multiplying \eqref{defX} on the left by $\psi$ and using \eqref{psiF}, one gets
\begin{equation}\label{uncalcullambda}
  \lambda_P(u,v) \psi X =v \psi GX , \qquad \forall (u,v)\in Conv(\Omega).
\end{equation}
From $\psi >0$ (see \eqref{defpsi}), $X>0$, \eqref{psiG>0} and \eqref{uncalcullambda}, one gets
\begin{equation}\label{lambda>0}
 \lambda_P(u,v) >0, \qquad \forall (u,v)\in Conv(\Omega).
\end{equation}
(For a different proof of \eqref{lambda>0}, see the proof of Lemma~\ref{lm:positivity}.)
From Lemma~\ref{lm:convex}, \eqref{lambda>0} and from the following linearity of the eigenvalue
\[\forall\alpha,u,v>0,\qquad\lb_P(\alpha u,\alpha v)=\alpha\lb_P(u,v),\]
we deduce that $(\bar u,\bar v)\in\Sigma.$
Moreover $(\bar u,\bar v)\not\in\bigl\{(u_{\min},r(u_{\min})),(u_{\max},r(u_{\max}))\bigr\}$ because we have assumed that $u_{\text{opt}}=\argmax_{u\in[u_\text{min},u_\text{max}]}\lb_P(u,r(u))\in(u_\text{min},u_\text{max}).$
\end{proof}

\begin{figure}

\begin{center}
\vspace{3cm}
\begin{minipage}{0.8\textwidth}
\setlength\unitlength{1cm}
\begin{picture}(12,6)
\put(0,0){\includegraphics[width=\textwidth]{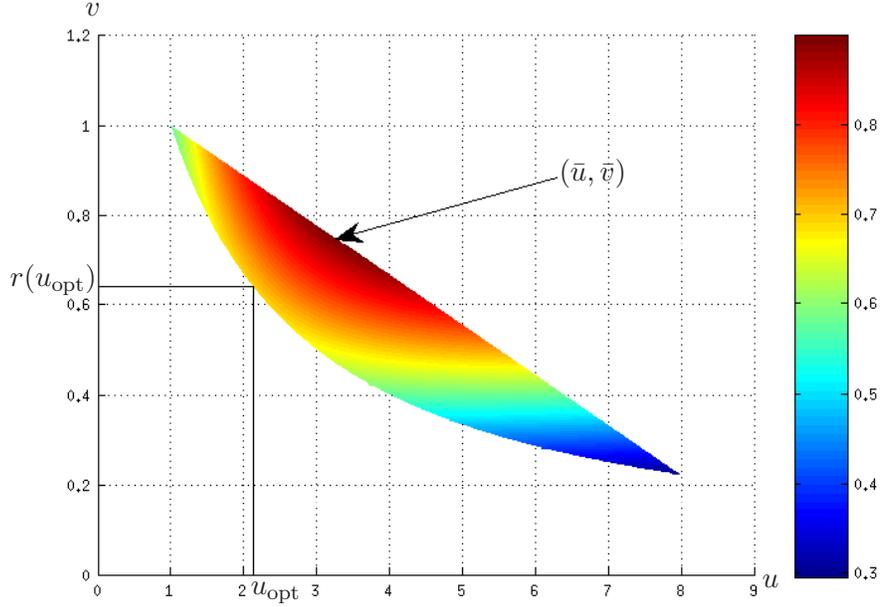}}
\put(.6,4.8){$r(u_\mathrm{opt})$}
\put(3.8,.7){$u_\mathrm{opt}$}
\put(7.9,6.2){$(\bar u,\bar v)$}
\put(10.6,.8){$u$}
\put(1.6,8.4){$v$}
\end{picture}
\end{minipage}
\end{center}

\caption{The eigenvalue function $\lb_P(u,v)$ on the convex hull $Conv(\Omega)$ for $r(u)=\f{2}{1+u},$ $u_{\min}=1,$ $u_{\max}=8,$ $n=3,$ $\tau_1=1,$ $\tau_2=10,$ $\beta_2=0.5,$ $\beta_3=1,$ $\kappa_{1,2}=2$ and $\kappa_{1,3}=\kappa_{2,3}=1.$}\label{fig:eigenvalue}

\end{figure}

We now prove that $(\bar u,\bar v)$, associated with accurate trajectories, satisfies the optimality condition~\eqref{eq:maximality}.

\begin{proposition}
\label{prop:exsolPerron}
Let $R>0$ and $S>0$. Then the constant control $(\bar u,\bar v)$ and the associated canonical direct and adjoint trajectories
\beq\label{eq:opt_cst_tra}\left\{\begin{array}{l}
\bar x(t)=R\bar X\,e^{\bar \lb_P t},\\
\bar p(t)=S\bar \phi\,e^{-\bar\lb_P t},
\end{array}\right.\eeq
satisfy the maximality condition
\begin{equation}
\label{maximality-condition}
H(\bar x(t),\bar p(t),\bar u,\bar v)=\max_{(u,v)\in Conv(\Omega)}H(\bar x(t),\bar p(t),u,v), \qquad \forall t\in[0,T].
\end{equation}
\end{proposition}

\begin{proof}[Proof of Proposition~\ref{prop:exsolPerron}]

Without loss of generality, we may assume that $R=S=1$.  From \eqref{defHamiltonien}, \eqref{defX}, \eqref{defphi} and \eqref{eq:opt_cst_tra}, we obtain, for every $t\in [0,T]$,
\begin{equation}
\label{Hbarubarv=}
H(\bar x(t),\bar p(t), \bar u,\bar v)=\lb_P(\bar u,\bar v).
\end{equation}
For any $(u,v)\in Conv(\Omega)$, we have, for every $t\in [0,T]$,
\begin{align}\label{barH}
H(\bar x(t),\bar p(t),u,v)=&\bar p(t)(uF+vG)\bar x(t)\nonumber\\
=& \bar p(t)(\bar u F+\bar v G)\bar x(t)+\bar p(t)\bigl((u-\bar u) F+(v-\bar v) G\bigr)\bar x(t)\nonumber\\
=&H(\bar x(t),\bar p(t),\bar u,\bar v) +\bar\phi\bigl((u-\bar u) F+(v-\bar v) G\bigr)\bar X.
\end{align}

Testing \eqref{defX} against the adjoint eigenvector $\phi$ and using the normalization $\phi X=1$ (see
\eqref{defphi}), we obtain
\begin{equation}\label{lambdaPuv}
\lambda_P(u,v)=\phi(uF+vG)X.
\end{equation}
Differentiating \eqref{lambdaPuv} with respect to $u$ and using \eqref{defX} together with \eqref{defphi}, we get
\begin{align}
\frac{\partial \lambda_P}{\partial u}&=\frac{\partial \phi}{\partial u}(uF+vG)X+\phi F X+
\phi (uF+vG)\frac{\partial X}{\partial u}\nonumber \\
&=\lambda_P(u,v)\frac{\partial \phi}{\partial u} X+\phi F X+\lambda_P(u,v)\phi \frac{\partial X}{\partial u}\nonumber\\
&=\lambda_P(u,v)\frac{\partial (\phi X)}{\partial u} +\phi F X\nonumber\\
&=\phi FX. \label{lambdau=}
\end{align}
We obtain in the same way that
\begin{equation}
\label{lambdav=}
\frac{\partial \lambda_P}{\partial v}=\phi GX.
\end{equation}
 From \eqref{Hbarubarv=}, \eqref{barH}, \eqref{lambdau=} and \eqref{lambdav=}, we obtain
\begin{equation}
\label{Huv=}
H(\bar x(t),\bar p(t),u,v)=\lb_P(\bar u,\bar v)+\left(\begin{array}{l}u-\bar u\\ v-\bar v\end{array}\right)\cdot\nabla\lb_P(\bar u,\bar v).
\end{equation}
Moreover, since $Conv(\Omega)$ is convex and $\lambda_P$ is maximal in $Conv(\Omega)$ at $(\bar u, \bar v)$,
\begin{equation}
\label{gradientuv}
\left(\begin{array}{l}u-\bar u\\v-\bar v\end{array}\right)\cdot\nabla\lb_P(\bar u,\bar v)\leq 0_, \qquad
\forall (u,v)\in Conv(\Omega).
\end{equation}
 From \eqref{Huv=} and \eqref{gradientuv}, we get that $(\bar u,\bar v)$ satisfies the maximality condition \eqref{maximality-condition}.

The proof is complete.
\end{proof}

\begin{remark} We easily check from \eqref{defX} and \eqref{defphi} that the trajectories~\eqref{eq:opt_cst_tra} are solutions to the direct and adjoint equations \eqref{eq:matrixformuv} and \eqref{eq:adjoint} with the constant control $(\bar u, \bar v).$ For these trajectories to satisfy additionally the initial condition in \eqref{eq:matrixformuv} and the terminal condition \eqref{eq:trans}, the initial distribution $x^0>0$ has to be taken collinear to $\bar X>0$ and the objective we want to maximize has to be modified by replacing in~\eqref{eq:cost2} the vector $\psi$ by a positive vector and collinear to $\bar\phi$.
\end{remark}

All the results of this section give indications that the optimal relaxed controls do not lie on $\Omega$, which would explain the oscillations that we observed numerically in Figure~\ref{fig:r_convex}.
In the next section, we precise these indications in the case of a two-compartment model.

\section{Dimension $n=2$}\label{secn=2}
As in the previous sections,  $u_{\min}$ and $u_{\max}$ are two positive real numbers such that $u_{\min}<u_{\max}$ and
$r\in\mathcal C^2([u_{\min},u_{\max}])$. We still assume that \eqref{rpositive} and \eqref{rsecondpositive} hold. However,
we no longer assume that \eqref{as:r3} holds. We precise what has been done in the two previous sections in the two dimensional case.

First we give the form of the matrices $F$ and $G$ in dimension 2 (see \eqref{defG}-\eqref{as:kappadiscret}):
\[
F=\left(
\begin{array}{cc}
 0 & 2\beta \\
 0 & -\beta
\end{array}
\right),\quad
G=
\left(
\begin{array}{cc}
 -\tau & 0\\
 \tau & 0
\end{array}
\right),
\]
with $\beta>0$ and $\tau>0$. Notice that, for the sake of clarity, we have skipped the indices of the coefficients:
the coefficient $\beta$ stands for $\beta_2$ and $\tau$ stands for $\tau_1$.

\

In dimension 2, the optimal control still lies on $\Sigma$ even if \eqref{as:r3} is no longer assumed to hold.
This is a consequence of the following lemma.
\begin{lemma}\label{lm:qpositivity}
For any control $(u,v)\in (L^\infty((0,T);(0,+\infty)))^2$, the solution $p=(p_1,p_2)$ to the adjoint equation
\[\dot p=-p(uF+vG),\qquad p(T)=\psi=(1,2),\]
satisfies
\beq\label{eq:pGpF}
(2p_1-p_2)(t)> 0\quad\text{and}\quad (p_2-p_1)(t)> 0,\quad \forall\, t\in[0,T).
\eeq
\end{lemma}

\begin{proof}[Proof of Lemma~\ref{lm:qpositivity}]
Denote by $\tilde p$ the vector
\[\tilde p:=\begin{pmatrix}\tilde p_1\\\tilde p_2\end{pmatrix}:=\begin{pmatrix}p_2-p_1\\2p_1-p_2\end{pmatrix}.\]
It satisfies the equation
\[\dot{\tilde p}=-\begin{pmatrix}-v\tau&u\beta\\2v\tau&-u\beta\end{pmatrix}\tilde p,\qquad \tilde p(T)=\begin{pmatrix}1\\0\end{pmatrix},\]
the result of the lemma follows.
\end{proof}

\subsection{The case $r(u_{\min})\leq r(u_{\max})$}

As a consequence of Lemma~\ref{lm:qpositivity}, we can solve very simply the optimal control problem in the case when $r(u_{\min})\leq r(u_{\max})$.

\begin{corollary}\label{co:theta>0}
If $r$ is such that $r(u_{\min})\leq r(u_{\max})$, then the optimal control is, for almost every $t\in [0,T]$, \[(u^*,v^*)\equiv(u_{\max},r(u_{\max})).\]
\end{corollary}

\begin{proof}[Proof of Corollary~\ref{co:theta>0}]
Let $x^*$  be the  trajectory corresponding to the control $(u^*,v^*)$. Let $p^*:\, [0,T]\rightarrow \R^2$ be the row vector solution of \eqref{eq:adjoint}-\eqref{eq:trans}.
When $r(u_{\min})\leq r(u_{\max})$, we have, for every $(u,v)\in Conv(\Omega)$,  $u\leq u_{\max}$ and $v\leq r(u_{\max})$.
From Lemma~\ref{lm:qpositivity}, we know that, for every time $t\in[0,T)$,
$$
p^*(t)Gx^*(t)=\tau(p^*_2(t)-p^*_1(t))x_1^*(t)>0\quad\text{and}\quad p^*(t)Fx^*(t)=\beta(2p_1^*(t)-p_2^*(t))x_2^*(t)>0.
$$
We conclude by using the maximality property \eqref{eq:maximality}.
\end{proof}

\subsection{The case $r(u_{\min})> r(u_{\max})$}

In this subsection we treat the case where $r(u_{\min})> r(u_{\max})$ which is more biologically relevant than the case $r(u_{\min})\leq r(u_{\max}),$ but also more delicate.
We start from a corollary of Lemma~\ref{lm:qpositivity} which ensures that the optimal control lies on the boundary $\Sigma$ of $Conv(\Omega).$

\begin{corollary}\label{co:sigma}
Let $(u^*,v^*)$ be an optimal control for problem \eqref{eq:cost2}. Then $(u^*(t),v^*(t))$
lies on $\Sigma$ for almost every $t\in[0,T]$.
\end{corollary}

\begin{proof}[Proof of Corollary~\ref{co:sigma}]
We use the same notations as in the proof of Corollary~\ref{co:theta>0}. 
Since $\Sigma$ is a string of the convex (see \eqref{rsecondpositive}) function $r$, we have that $v\leq\sigma(u)$ for every $(u,v)\in Conv(\Omega)$. From Lemma~\ref{lm:qpositivity}, we deduce that, for every time $t\in[0,T)$,
$$
p^*(t)Gx^*(t)=\tau(p_2^*(t)-p_1^*(t))x_1^*(t)>0.
$$
Then, using the maximality property \eqref{eq:maximality}, we conclude  that $v^*(t)=\sigma(u^*(t))$ for almost every $t\in [0,T]$.
\end{proof}

Using Corollary~\ref{co:sigma}, we can reduce the optimal control problem \eqref{eq:cost2} to the control set $\Sigma$. Then $v=\sigma(u)=\theta u+\zeta$
with $\theta$ defined in \eqref{deftheta} and $\zeta$ defined in \eqref{defzeta}.
Notice that the assumption $r(u_{\max})<r(u_{\min})$ ensures that
\[\theta<0 \quad\text{ and }\quad \zeta>0.\]
The dynamic equation \eqref{eq:matrixformuv} and the adjoint equation \eqref{eq:adjoint} become respectively
\begin{gather}
\label{eq:dynx}
\dot x=u(F+\theta G)x+\zeta Gx,
\\
\label{eq:dynp}
\dot p=-up(F+\theta G)-\zeta pG,
\end{gather}
and the relaxed optimal control problem~\eqref{eq:cost2} is replaced by
\beq
\label{eq:cost3}
\text{maximize } J(u)=\psi x(T),\, u:[0,T]\rightarrow [u_{\min},u_{\max}]  \text{ is a Lebesgue measurable function},
\eeq
subject to dynamics~\eqref{eq:dynx} with $x(0)=x^0>0.$
Call
\[H_\Sigma(x,p,u)=up(F+\theta G)x+\zeta pGx\]
the Hamiltonian for dynamics~\eqref{eq:dynx}-\eqref{eq:dynp} and define by
\[\Phi(x,p):=\f{\p H_\Sigma(x,p,u)}{\p u}=p(F+\theta G)x\]
the switching function.
For $(x,p)$ solution to~\eqref{eq:dynx}-\eqref{eq:dynp}, we also call switching function the quantity
\beq\label{def:Phi}\Phi(t):=\Phi(x(t),p(t)).\eeq
The maximum condition of the PMP writes for problem~\eqref{eq:cost3}
\beq\label{eq:maximality2}H_\Sigma(x^*(t),p^*(t),u^*(t))=\max_{u\in[u_{\min},u_{\max}]} H_\Sigma(x^*(t),p^*(t),u)\eeq
for almost every $t\in[0,T],$
and it is verified for $(x^*,p^*,u^*)$ if and only if $\Phi^*(t)=\Phi(x^*(t),p^*(t))$ satisfies
\begin{equation}\label{uaf}
\Phi^*(t)=0\quad\text{or}\quad\left\{\begin{array}{ll}\Phi^*(t)>0&\text{when}\ u^*(t)=u_{\max},\\
\Phi^*(t)<0&\text{when}\ u^*(t)=u_{\min},
\end{array}\right.
\end{equation}
for almost every $t\in[0,T].$
First we look for singular trajectories on open intervals $I$, {\it i.e.,} $(x,p,u)$ with $x\in \mathcal{C}^0(I;(0,+\infty)^2)$, $p^T\in \mathcal{C}^0(I;(0,+\infty)^2)$, $u\in L^\infty(I;[u_{\min},u_{\max}])$ solutions of \eqref{eq:dynx}-\eqref{eq:dynp}
such that
\[\Phi(t)=0,\qquad \text{for almost every }t\in I.\]

\begin{theorem}\label{th:singular}
 For a nonempty open interval $I$, $t\in I\mapsto (x(t),p(t),u(t))$ is a singular trajectory if and only if
\begin{gather}
\label{eq:using}u(t)= u_{\mathrm{sing}}:=\frac{\zeta\tau}{\sqrt{-2\theta\tau\beta}}\,
\frac{2\beta+\sqrt{-2\theta\tau\beta}}{\beta+2\sqrt{-2\theta\tau\beta}-\theta\tau},\qquad  \text{for almost every }t\in I,
\end{gather}
and there exist  two positive real numbers $R$ and $S$ such that
\begin{gather}
\label{expressionxp}
x(t)=RXe^{\lb t}\qquad\text{and}\qquad p(t)=S\phi\, e^{-\lb t}, \qquad \forall t\in I,
\end{gather}
where
\begin{itemize}
\item $\lb$  is the Perron eigenvalue of the matrix  $u_\mathrm{sing}(F+\theta G)+\zeta G$ and
\begin{gather}
\label{valuelambdasing}
\lb=\f{\zeta\tau\beta}{\beta+2\sqrt{-2\theta\tau\beta}-\theta\tau},
\end{gather}
\item $X$ and $\phi$ are respectively direct and adjoint positive eigenvectors of the matrix $u_\mathrm{sing}(F+\theta G)+\zeta G$ associated to the Perron eigenvalue $\lb$.
\end{itemize}
\end{theorem}

\begin{proof}[Proof of Theorem~\ref{th:singular}]
Let us first remark that $u_{\mathrm{sing}}$ defined by \eqref{eq:using} satisfies
\[u_\mathrm{sing}=\f\zeta{-\theta}
\f{\sqrt{-2\theta\tau\beta}-\theta\tau}{\beta+2\sqrt{-2\theta\tau\beta}-\theta\tau}
\in\Bigl(0,\f{\zeta}{-\theta}\Bigr).\]
thus, $u_\mathrm{sing}\theta+\zeta>0$ and the matrix $u_\mathrm{sing}F+(u_\mathrm{sing}\theta+\zeta) G$ satisfies the assumptions of the Perron-Frobenius theorem.

{\it First step: ``If'' part.}
Simple computations prove that
\begin{gather}
\label{defX2}
X=(2\sqrt\beta,\sqrt{-2\theta\tau})^T,\\
\label{defphi2}
\phi=(\sqrt\beta+\sqrt{-2\theta\tau},2\sqrt\beta+\sqrt{-2\theta\tau})
\end{gather}
are, respectively, right and left eigenvectors of the matrix $u_\mathrm{sing}(F+\theta G)+\zeta G$ associated to the eigenvalue $\lambda$ defined by \eqref{valuelambdasing}. Since $X>0$ and $\phi>0$ are positive, they are necessarily Perron eigenvectors of this matrix and $\lb$ is its Perron eigenvalue $\lb_P(u_\mathrm{sing})$.

Let $\bar x: I\rightarrow \R^2$, $\bar p : I\rightarrow \R^2$, $\bar u : I \rightarrow \R$ be defined by
\begin{gather}
\label{defbarxbarp}
\bar x(t):= e^{\lambda t}X,\qquad \bar p(t):= e^{-\lambda t}\phi\qquad
\text{ and }\qquad \bar u(t)= u_\mathrm{sing},\qquad \forall t\in I.
\end{gather}
As already used in the previous section, $(\bar x, \bar p, \bar u)$ are solutions of \eqref{eq:dynx}-\eqref{eq:dynp}. It remains
only to check that along this trajectory $(\bar x, \bar p, \bar u)$,  $\Phi=0$, which holds since $\Phi=\phi(F+\theta G)X=0.$
Notice that because $\phi(F+\theta G)X=\lb_P'(u_\mathrm{sing})\,\phi X$, we also get that $u_\mathrm{sing}$ is a critical point of $\lb_P$.

\

{\it Second step: ``Only if'' part.}
Suppose that $(x,p,u)$ is a singular trajectory on an open interval $I$. We have $\Phi(t)=p(t)(F+\theta G)x(t)=0$ on $I$.
This gives the relation
\beq\label{eq:phi=0}-\theta\tau x_1(p_1-p_2)+\beta x_2(2p_1-p_2)=0.\eeq
Differentiating $\Phi$ with respect to $t$ on $I$, we get
\[\dot\Phi=\zeta p[F,G]x=0,\]
where $[F,G]:=FG-GF$ is the Lie bracket of $F$ and $G$.
It provides a second identity
\beq\label{eq:dphi=0}x_1(2p_1-p_2)+2x_2(p_1-p_2)=0.\eeq
If we differentiate $\Phi$ a second time, we get
\beq
\label{eq:ddphi=0}
\ddot\Phi=
\zeta u p\bigl([\, [F,G],F\, ]+\theta [\, [F,G],G\, ]\bigr) x+\zeta^2 p[\, [F,G],G\, ]x=0.
\eeq
Using~\eqref{eq:dphi=0}, we obtain
\begin{align*}
p([\, [F,G],F\, ])x= & 4\tau\beta^2 p_1x_2,\\
p([\, [F,G],G\, ])x=& -2\tau^2\beta p_2 x_1.
 \end{align*}
 We remark that
 $$
 p\bigl([\, [F,G],F\, ]+\theta ([\, [F,G],G\, ]\bigr) x=\frac{\zeta p}{ u}[\, [F,G],G\, ]x=-\frac{2\zeta}{u}\tau^2\beta p_2 x_1
 $$ cannot vanish because $x>0$ and $p>0$.
So we can divide~\eqref{eq:ddphi=0} by this term and we get
\beq\label{eq:using1}u=\f{-\zeta p([\, [F,G],G\, ])x}{p([\, [F,G],F\, ])x+\theta p([\, [F,G],G\, ])x}=\f{\zeta \tau p_2x_1}{2\beta p_1 x_2-\theta \tau p_2 x_1}.\eeq
Consider now~\eqref{eq:phi=0}-\eqref{eq:dphi=0} as a system of equations for the unknown $(x_1,x_2)$.
Since $x$ is positive, this system must have a vanishing determinant and it gives the relation
\[\beta (2p_1-p_2)^2+2\theta \tau(p_1-p_2)^2=0.\]
Using~\eqref{eq:pGpF} in Lemma~\ref{lm:qpositivity}, we can write
\[2p_1-p_2= \sqrt{\f{-2\theta\tau}{\beta}}(p_2-p_1),\]
and finally we get
\beq\label{eq:p1p2}\f{p_1}{p_2}=\f{\sqrt{\beta}+\sqrt{-2\theta \tau}}{2\sqrt{\beta}+\sqrt{-2\theta \tau}}.\eeq
Similarly, if we consider \eqref{eq:phi=0}-\eqref{eq:dphi=0} as a system of equations for the positive unknown $(p_1,p_2)$, from the fact that the determinant vanishes, we obtain
\[\theta \tau x^2_1+2\beta x^2_2=0.\]
Since $x>0$, $\theta<0$ and $\tau >0$, we deduce that
\beq\label{eq:x1x2}x_1=\sqrt{\f{2\beta}{-\theta\tau}}x_2.\eeq
Plugging~\eqref{eq:p1p2} and~\eqref{eq:x1x2} into~\eqref{eq:using1}, we get~\eqref{eq:using}. From \eqref{defX2}, \eqref{defphi2},
\eqref{defbarxbarp}, \eqref{eq:p1p2} and \eqref{eq:x1x2},
one gets the existence of $R: I\rightarrow (0,+\infty)$ and $S: I\rightarrow (0,+\infty)$ such that
\[x=R\bar x \quad \text{ and }\quad  p=S\bar p.\]
Using the fact that $(x,p,u_\mathrm{sing})$ and $(\bar x,\bar p,u_\mathrm{sing})$ are both solutions of \eqref{eq:dynx}-\eqref{eq:dynp}, one readily gets that
$\dot R=\dot S =0$.

The proof is complete.
\end{proof}

In the proof of Theorem~\ref{th:singular}, we have pointed out a link between the singular trajectories and the critical points of the Perron eigenvalue.
In Theorem~\ref{th:perrondim2}, we prove that $u_\mathrm{sing}$ is actually the unique maximum of $\lb_P$.

\begin{theorem}\label{th:perrondim2}
The Perron eigenvalue $\lb_P(u)$ of the matrix $u(F+\theta G)+\zeta G$ is well defined on  $(0,\displaystyle\frac{\zeta}{-\theta})$ and it reaches its unique maximum at $\bar u=u_\mathrm{sing}$.
\end{theorem}

\begin{proof}[Proof of Theorem~\ref{th:perrondim2}]

The function $u\mapsto \theta u+\zeta$ is positive on $(0,\frac{\zeta}{-\theta})$, so the matrix $uF+(\theta u+\zeta)G$ satisfies the hypotheses of the Perron-Frobenius theorem on this interval.
The function $u\mapsto\lb_P(u)$ is positive for $u\in(0,\frac{\zeta}{-\theta})$ (see~\eqref{lambda>0}) and tends to zero as $u$ tends to $0$ or to $\frac{\zeta}{-\theta} .$
So $\lb_P$ necessarily has a maximum on $(0,\frac{\zeta}{-\theta} )$.

Every critical point of  $\lb_P$ provides with \eqref{expressionxp} a singular trajectory.
Since Theorem~\ref{th:singular} ensures that there exists a unique singular trajectory (up to multiplicative constants for $x$ and $p$), this gives the uniqueness of the critical point of $\lb_P.$ Therefore this critical point realizes the maximum of $\lb_P$ on $(0,\frac{\zeta}{-\theta})$ and the proof is complete.

\end{proof}

\begin{remark}
The proof of Theorem~\ref{th:perrondim2} can be done by explicit computations, without using the uniqueness of the singular control.
This instructive and useful computational proof is given in Appendix~\ref{sec:altproof}.
\end{remark}

As an immediate consequence of the proofs of Theorems~\ref{th:singular} and~\ref{th:perrondim2}, we have explicit expressions of the optimal eigenelements.

\begin{corollary}\label{co:opt_perron}
The maximal Perron eigenvalue is
\[\bar\lb_P=\lb_P(\bar u)=\f{\zeta\tau\beta}{\beta+2\sqrt{-2\theta\tau\beta}-\theta\tau}\]
and the associated right and left eigenvectors are given by
\[\bar X=(2\beta,\sqrt{-2\theta\tau\beta})^T\qquad\text{and}\qquad
\bar \phi=(\beta+\sqrt{-2\theta\tau\beta},2\beta+\sqrt{-2\theta\tau\beta}).\]
\end{corollary}

We are now ready to exhibit the unique optimal control when the horizon $T$ is large enough.
It consists mainly in the singular control $u_\mathrm{sing}$, except in the regions
close to the endpoints of $[0,T]$. For small time, the optimal control depends
on the initial data $x^0$ and it consists in reaching as fast as possible
the singular trajectory. Then the control remains constant equal to the value
which maximizes the Perron eigenvalue (see Figure~\ref{fig:relax}).
At the end of the experiment, the control is $u_{\min}$ due to the
transversality condition induced by the objective function.
This kind of strategy is known as \emph{turnpike properties}
(see~\cite{TrelatZuazua,Zaslavski} for instance), the ``turnpike'' is
the singular trajectory which corresponds to the Perron
eigenvector $X$ with an exponential growth of maximal rate $\lb_P(\bar u)$.

We divide the construction of the optimal control in two steps. First,
we build a control such that the PMP is satisfied (Theorem \ref{th:part_cont}). Then, with an analysis of the switching function $\Phi(t)$, we prove that this is the only possible one (Theorem \ref{th:uniqueness}).

\

Before stating the results, let us define projective variables.
 For $x=(x_1,x_2)^T$ and $p=(p_1,p_2)$, where $x_1$, $x_2$, $p_1$ and $p_2$ are positive real numbers, we define
\[y:=\f{x_1}{x_1+x_2}\in [0,1]\quad
  \text{ and }\quad q:=\f{p_1}{p_1+p_2}\in [0,1].\]
 For $(x(t),p(t))$ a solution to~\eqref{eq:dynx}-\eqref{eq:dynp},  the projections $y(t)$ and $q(t)$ satisfy the dynamics
\begin{gather}
\label{doty=}
\dot y = \mathcal{Y}(y,u):=2u\beta-(3u\beta+(u\theta+\zeta)\tau)y+u\beta y^2,
\\
\label{dotq=}
\dot q=\mathcal{Q}(q,u):=-\bigl[(\theta u+\zeta)\tau-(3(\theta u+\zeta)\tau-u\beta)q+(2(\theta u+\zeta)\tau-3u\beta)q^2\bigr].
\end{gather}
In our problem the initial condition $x(0)=x^0=(x^0_1,x^0_2)^T>0$ and the terminal condition $p(T)=\Psi=(1,2)>0$ guarantee that $x(t)$ and $p(t)$ are positive for all $t\in[0,T].$
So $y(t)$ and $q(t),$ which satisfy
\[y(0)=\f{x_1^0}{x_1^0+x_2^0}\qquad\text{and}\qquad \dis q(T)=\f13,\]
are well defined and belong to $(0,1)$ for every time $t\in[0,T].$
Finally for $u$ in $[u_{\text{min}},u_{\text{max}}]$, let $Y(u)\in(0,1)$ be the projection of the Perron eigenvector $X(u)$ of the matrix $uF+(\theta u+\zeta) G$ and let $\pi(u)$ be the projection of the adjoint Perron eigenvector $\phi(u)$.

\begin{theorem}\label{th:part_cont}
There exist a time $T_\psi>0$ and a function $T_0$ defined on $[0,1]$, satisfying
\[\forall y\in[0,1]\setminus\{Y(\bar u)\},\ T_0(y)>0,\qquad T_0(Y(\bar u))=0,\qquad\sup_{[0,1]}T_0<+\infty,\]
such that, for $T>T_0(y(0))+T_\psi$, the control defined by
\beq\label{eq:upart}u^*(t)=\left\{\begin{array}{ll}
	    \left\{ \begin{array}{l}
	     u_{\min}\ \text{if}\ y(0)>Y(\bar u)\\
	     u_{\max}\ \text{if}\ y(0)<Y(\bar u)
	     \end{array}\right.
              & \text{for}\ t\in[0,T_0(y(0))],
	     \vspace{1mm}\\
               \ \bar u & \text{for}\ t\in(T_0(y(0)),T-T_\psi], \\
               \ u_{\min} & \text{for}\ t\in(T-T_\psi,T],
              \end{array}\right. \eeq
satisfies the maximality property~\eqref{eq:maximality2}.
Moreover the corresponding trajectories $y^*(t)$ and $q^*(t)$ satisfy $y^* \equiv Y(\bar u)$ and $q^*\equiv \pi(\bar u)$ on $[T_0(y(0)),T-T_\psi]$.
\end{theorem}

Based on Theorem~\ref{th:part_cont}, we can state our main theorem.
\begin{theorem}\label{th:uniqueness}
There exists a time $T_c>T_\psi+\displaystyle\sup_{[0,1]}T_0$ such that, for every final time $T>T_c$, the control defined by~\eqref{eq:upart} is the unique solution to the optimal relaxed control problem~\eqref{eq:cost3}.
\end{theorem}

For the optimal relaxed control $(u^*,\sigma(u^*))\to\Sigma$ given by~\eqref{eq:upart}, we can easily build a sequence of piecewise constant functions $(u^n,v^n)_{n\in\N}\to\Omega$ which converges weakly to $(u^*,\sigma(u^*)),$ as claimed in~\eqref{weakconvergence}.
It suffices to replace $u^*=\bar u$ on the interval $[T_0,T-T_\psi)$ by
\[u^n(t)=\left\{\begin{array}{ll}
u_{\min}&\text{if}\quad t\in[T_0+(k-1)\Delta t,T_0+(k-1)\Delta t+\Delta_{\min})
\vspace{1mm}\\
u_{\max}&\text{if}\quad t\in[T_0+(k-1)\Delta t+\Delta_{\min},T_0+k\Delta t)\end{array}\right.,
\quad\text{for}\quad k=1,\cdots,n,\]
where we have defined
\beq\label{eq:deltat}\Delta t=\f{T-(T_0+T_\psi)}{n},\qquad\Delta_{\min}:=\f{u_{\max}-\bar u}{u_{\max}-u_{\min}}\,\Delta t\qquad\text{and}\qquad \Delta_{\max}:=\f{\bar u-u_{\min}}{u_{\max}-u_{\min}}\,\Delta t,\eeq
and to set $v^n=r(u^n).$
This sequence oscillates more when $n$ increases, keeping the same mean values \[\int_0^Tu^n(t)\,dt=\int_0^Tu^*(t)\,dt\qquad\text{and}\qquad\int_0^Tv^n(t)\,dt=\int_0^Tv^*(t)\,dt.\]
This is what happens in Figure~\ref{fig:r_convex} when we solve numerically a discretized optimal control problem.

\begin{figure}[h!]
\begin{center}
\begin{minipage}[b]{0.49\linewidth}
\includegraphics [width=\linewidth]{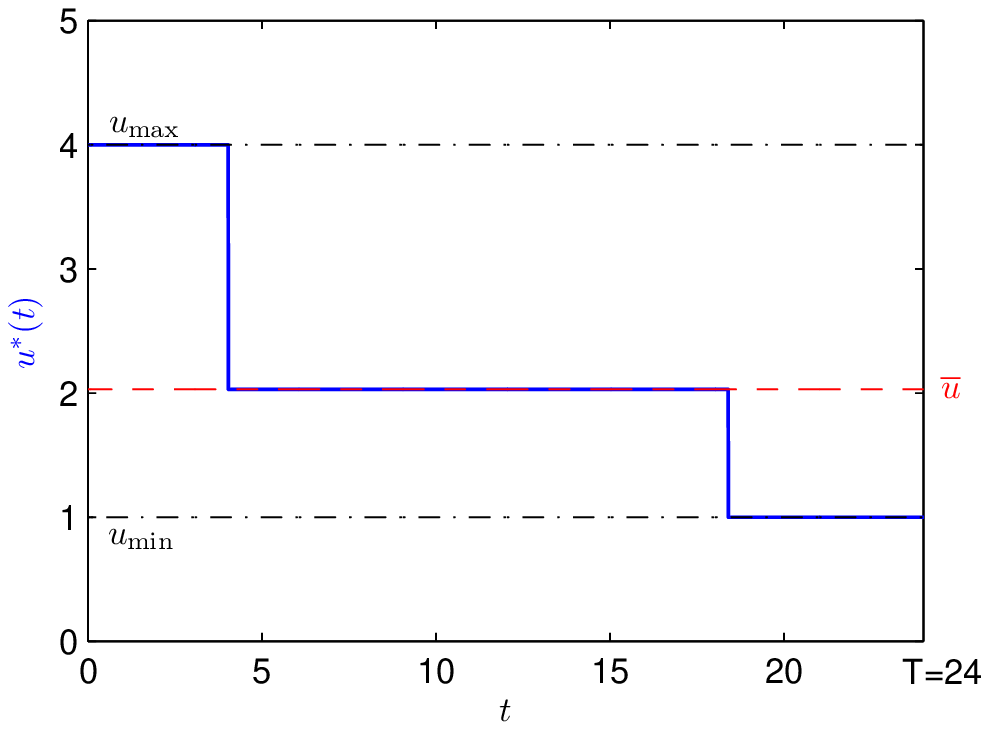}
\end{minipage}\hfill
\begin{minipage}[b]{0.49\linewidth}
\includegraphics [width=\linewidth]{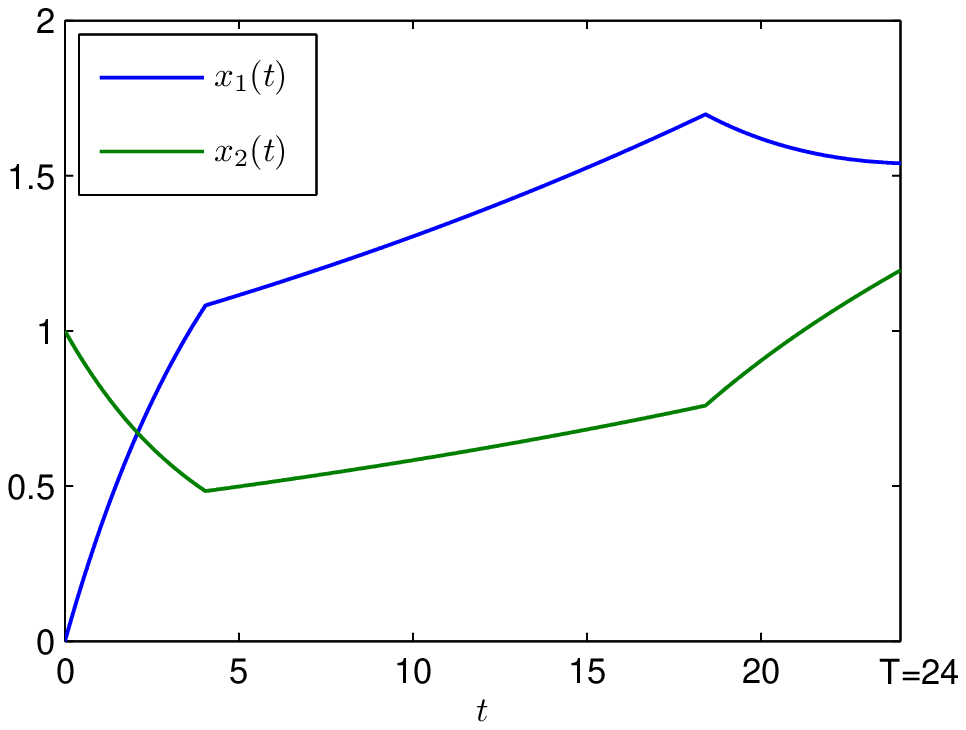}
\end{minipage}
\end{center}

\begin{center}
\begin{minipage}[b]{0.49\linewidth}
\includegraphics [width=\linewidth]{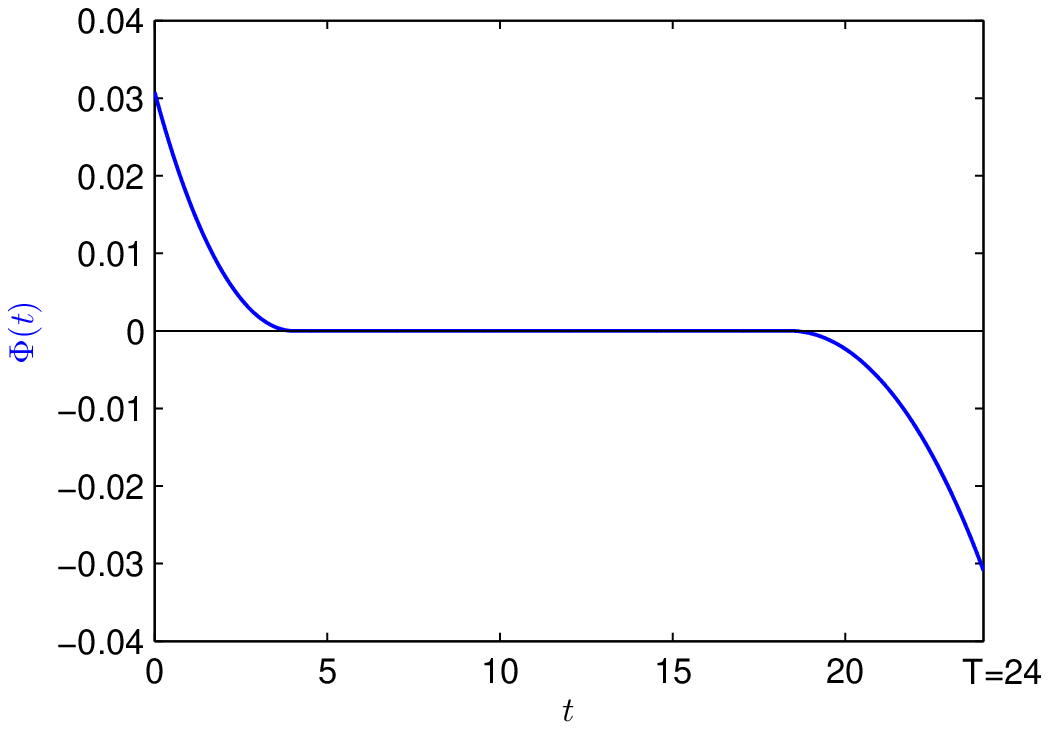}
\end{minipage}\hfill
\begin{minipage}[b]{0.49\linewidth}
\includegraphics [width=\linewidth]{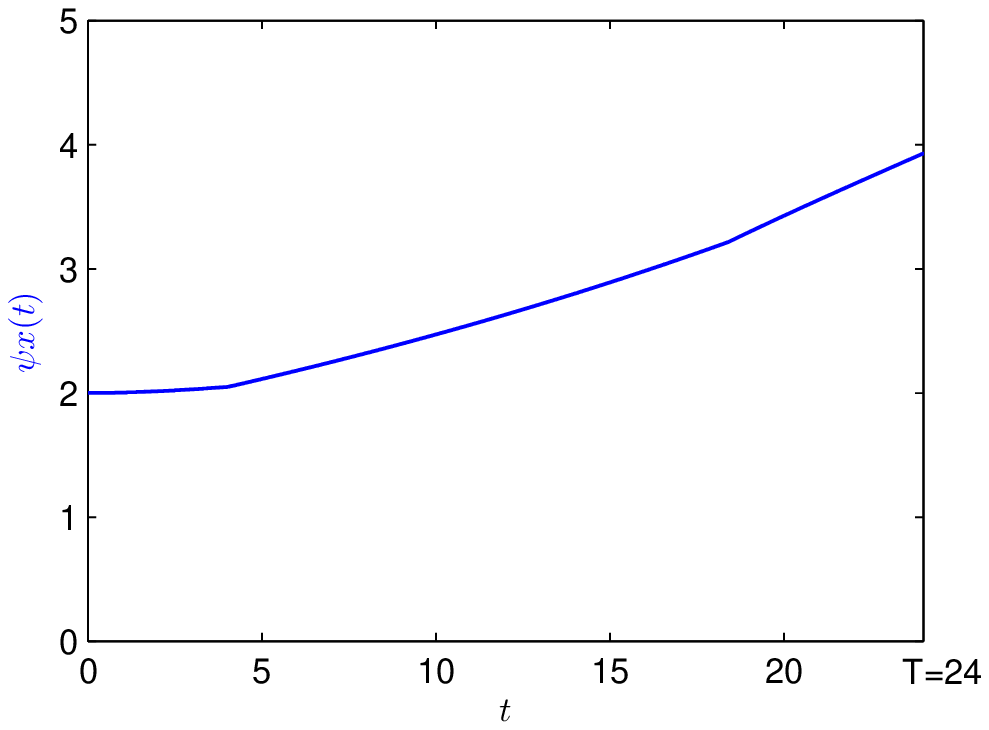}
\end{minipage}
\end{center}
\caption{\label{fig:relax}Top left: the optimal control for $T=24,$ $n=2,$ $\theta=-0.2,$ $\zeta=1,$ $\tau=0.1,$ $\beta=0.05$ and the initial data $x^0=(0,1)^T.$
Top right: the corresponding trajectories $x_1(t)$ and $x_2(t).$
Bottom left: the switching function $\Phi(t).$
Bottom right: the evolution of the objective $\psi x(t)=x_1(t)+2x_2(t).$}
\end{figure}

\

Before proving Theorems~\ref{th:part_cont} and~\ref{th:uniqueness}, we give some useful preliminary results.
First note that for $u$ fixed in $[u_\text{min},u_\text{max}],$ $Y(u)$ and $\pi(u)$ are the unique steady states of respectively~\eqref{doty=} and~\eqref{dotq=} in the interval $[0,1].$
Moreover we know the sign of the vector fields.
\begin{lemma}\label{lm:signdotyq}
For every $u$ in $[u_{\text{min}},u_{\text{max}}]$,
\begin{align}
\label{signdoty}
(y-Y(u))\mathcal{Y}(y,u)&<0,\qquad \forall y \in [0,1] \setminus\{Y(u)\},
\\
\label{signdotq}
(q-\pi(u))\mathcal{Q}(q,u)&>0,\qquad \forall q \in [0,1] \setminus\{\pi (u)\}.
\end{align}
\end{lemma}

\begin{proof}
Let $u\in[u_{\text{min}},u_{\text{max}}].$
The function $y\mapsto(y-Y(u))\mathcal{Y}(y,u)$ is a third order polynomial which vanishes only at $y=Y(u)$ on the interval $[0,1].$
We obtain~\eqref{signdoty} by computing for instance $\mathcal Y(0,u)=2u\beta>0$ and $\mathcal Y(1,u)=-(u\theta+\zeta)\tau<0.$
The same argument allows to prove~\eqref{signdotq}.
\end{proof}

\begin{corollary}\label{co:asymptotic}
For $u$ fixed in $[u_{\text{min}},u_{\text{max}}]$, every solution $y$ of \eqref{doty=} with $y(0)$ in $[0,1]$ satisfies
\[(y(t)-Y(u))^2\downarrow0\quad\text{when}\quad t\to+\infty,\]
and every solution $q$ of \eqref{dotq=} with $q(0)$ in $[0,1]$ satisfies
\[(q(t)-\pi(u))^2\downarrow0\quad\text{when}\quad t\to-\infty.\]
\end{corollary}

\begin{proof}
We get from~\eqref{signdoty} that
\[\f{d}{dt}(y(t)-Y(u))^2=2(y(t)-Y(u))\mathcal Y(y(t),u)<0\]
while $y(t)\neq Y(u),$
and from~\eqref{signdotq} that
\[\f{d}{dt}(q(-t)-\pi(u))^2=-2(q(-t)-\pi(u))\mathcal Q(q(-t),u)<0\]
while $q(t)\neq \pi(u).$
\end{proof}

\begin{remark}
Corollary~\ref{co:asymptotic} allows to recover the property that $y(t)\in(0,1)$ and $q(t)\in(0,1)$ for every $t\in[0,T]$ as soon as $y(0)\in(0,1)$ and $q(T)\in(0,1).$
\end{remark}

From the proof of Corollary~\ref{co:asymptotic} we readily get another useful consequence of Lemma~\ref{lm:signdotyq}.

\begin{corollary}\label{co:defT}
For $u$ fixed in $[u_{\text{min}},u_{\text{max}}],$ every solution $y$ of \eqref{doty=} satisfies
\[\left\{\begin{array}{l}
0<y(0)<Y(u)\implies\exists !\, T>0,\ y(-T)=0,
\vspace{2mm}\\
Y(u)<y(0)<1\implies\exists !\, T>0,\ y(-T)=1.\end{array}\right.\]
\end{corollary}

The following lemma orders the steady states $Y(u)$ and $\pi(u)$ for different values of $u.$

\begin{lemma}\label{lm:orderYpi}
We have the comparisons
\[Y(u_\text{min})<Y(\bar u)<Y(u_\text{max})\qquad\text{and}\qquad\pi(u_\text{min})>\pi(\bar u)>\pi(u_\text{max})>\f13.\]
\end{lemma}

\begin{proof}
We study the variations of the function $u\mapsto Y(u)$.
Starting from the computations in the computational proof of Theorem~\ref{th:perrondim2} in Appendix~\ref{sec:altproof}, we get
\[Y(u)=\f{\lb_P(u)+u\beta}{\lb_P(u)+u\beta+(u\theta+\zeta)\tau}=\f{u\beta-(u\theta+\zeta)\tau+\sqrt{\Delta}}{u\beta+(u\theta+\zeta)\tau+\sqrt{\Delta}},\]
and by differentiation
\[Y'(u)=\f{2h(u)}{\bigl(u\beta+(u\theta+\zeta)\tau+\sqrt{\Delta}\bigr)^2}\qquad\text{with}\qquad h(u):=\beta\zeta\tau-\theta\tau\sqrt{\Delta}+\f{\Delta'}{2\sqrt{\Delta}}(u\theta+\zeta)\tau.\]
To have the sign of $h$ on $[0,\displaystyle\frac{\zeta}{-\theta} ]$, we differentiate $h$ with respect to $u$
\[h'(u)=\frac{(u\theta+\zeta)\tau}{4\Delta\sqrt{\Delta}}\left(2\Delta\Delta''-(\Delta')^2\right).\]
We know from Lemma~\ref{lm:lbsecond} in Appendix~\ref{sec:altproof} that $2\Delta\Delta''-(\Delta')^2<0.$
So $h$ is decreasing and since $h(\displaystyle\frac{\zeta}{-\theta})=2\zeta\beta\tau>0$, we get
that $h>0$ on $[0,\displaystyle\frac{\zeta}{-\theta}]$.
Then $Y$ is increasing and
\[Y(u_{\min})<Y(\bar u)<Y(u_{\max}).\]

For the variations of $\pi(u),$ we compute
\[\pi(u)=\f{\lb_P(u)+u\beta}{\lb_P(u)+3u\beta}=\frac{u\beta-(u\theta+\zeta)\tau+
\sqrt{\Delta}}{5u\beta-(u\theta+\zeta)\tau+\sqrt{\Delta}}\]
and
\[\pi'(u)=\frac{4 g(u)}{\left(5u\beta-(u\theta+\zeta)\tau+\sqrt{\Delta}\right)^2}\qquad\text{with}\qquad
g(u)=\beta\left(\zeta\tau-\sqrt\Delta+\f{u \Delta'}{2\sqrt\Delta}\right).\]
Using Lemma~\ref{lm:lbsecond} in Appendix~\ref{sec:altproof} we have
\[g'(u)=\f{\beta}{4\Delta\sqrt\Delta}\left(2\Delta\Delta''-(\Delta')^2\right)<0,\]
so $g<0$ since $g(0)=0$.
We obtain that $\pi(u)$ is decreasing and
\[\pi(u_{\min})>\pi(\bar u)>\pi(u_{\max})>\pi\Bigl(\f{\zeta}{-\theta}\Bigr)=\frac{1}{3}.\]

\end{proof}

We give now a last lemma and a corollary, which are key points in the proofs of Theorem~\ref{th:part_cont} and Theorem~\ref{th:uniqueness}.

\begin{lemma}\label{lm:Phidot=}
The switching function $\Phi$ defined in~\eqref{def:Phi} satisfies the equation
\[\dot\Phi
=\Lambda_1\Phi+\Lambda_2,\]
with
\[\Lambda_1 :=\zeta\tau\frac{x_1}{x_2}\qquad\text{and}\qquad
\Lambda_2 := \zeta\tau(p_2-p_1)\Bigl(-\theta\tau\Bigl(\frac{x_1}{x_2}\Bigr)^2-2\beta\Bigr)x_2.\]
Moreover for every $t\in[0,T]$ we have $\Lambda_1(t)>0$ and
\[\text{sign}(\Lambda_2(t))=\text{sign}(y(t)-Y(\bar u)),\]
where, for every real number $a$, $\text{sign}(a)=1$ if $a>0$, $\text{sign}(a)=-1$ if $a<0$ and $\text{sign}(a)=0$ if $a=0$.
\end{lemma}

\begin{proof}
From the positivity of $x(t),$ we get $\Lambda_1(t)>0.$
For $\Lambda_2$, it is noteworthy from Lemma~\ref{lm:qpositivity} that $p_2-p_1>0$.
Thus, the sign of $\Lambda_2$ does not depend on $p$ and from Corollary~\ref{co:opt_perron}, we get
\[\text{sign}(\Lambda_2)=\text{sign}\left(\f{x_1^2}{x_2^2}-\f{2\beta}{-\theta\tau}\right)=\text{sign}\left(\biggl(\f{y}{1-y}\biggr)^2-\biggl(\f{Y(\bar u)}{1-Y(\bar u)}\biggr)^2\right)=\text{sign}(y-Y(\bar u)).\]
\end{proof}

\begin{corollary}\label{co:signPhi}
For all $t>0,$ we have
\[\left\{\begin{array}{l}
\Bigl(\Phi^*(t)\geq0\quad\text{and}\quad y^*(t)<Y(\bar u)\Bigr)\implies\Bigl(\Phi^*(s)>0,\quad\forall s<t\Bigr),
\vspace{2mm}\\
\Bigl(\Phi^*(t)\leq0\quad\text{and}\quad y^*(t)>Y(\bar u)\Bigr)\implies\Bigl(\Phi^*(s)<0,\quad\forall s<t\Bigr).
\end{array}\right.\]
\end{corollary}

\begin{proof}
Assume that $\Phi^*(t)\geq0$ and $y^*(t)<Y(\bar u).$
If $\Phi^*(t)=0, $ then Lemma~\ref{lm:Phidot=} gives $\dot\Phi^*(t)<0$ so there exists $t'<t$ such that $\Phi^*>0$ on $(t',t).$
As a consequence $u^*\equiv u_{\max}$ on $(t',t)$ which ensures by using Lemma~\ref{lm:signdotyq} and Lemma~\ref{lm:orderYpi} that $y^*(t')<Y(\bar u).$
So we can restrict to the case $\Phi^*(t)>0,\ y^*(t)<Y(\bar u)$ and we suppose by contradiction that
\beq\label{contradiction}\exists t_0<t,\ \Phi^*(t_0)=0\ \text{and}\ \Phi^*(s)>0,\ \forall s\in(t_0,t).\eeq
Using Lemma~\ref{lm:signdotyq} and Lemma~\ref{lm:orderYpi}, we obtain that $y^*<Y(\bar u)$ on $[t_0,t].$
By Lemma~\ref{lm:Phidot=} we get that $\Lambda_2<0$ and then $\dot\Phi^*<\Lambda_1\Phi^*$ on $[t_0,t].$
By Gr\"onwall's inequality we deduce that $\Phi^*(t_0)>\Phi^*(t)\,e^{-\int_{t_0}^t\Lambda_1(s)\,ds}>0.$
This contradicts~\eqref{contradiction} so the first implication of Lemma~\ref{co:signPhi} is proved.
The proof of the second implication follows from a similar argument.
\end{proof}

\

In the proofs of Theorem~\ref{th:part_cont} and Theorem~\ref{th:uniqueness}, we will use the following compact definitions: a triplet $(y,u,y^0)$ (resp. $(q,u,q^0)$) is said to be solution to~\eqref{doty=} (resp. to~\eqref{dotq=}) if $y$ satisfies Equation~\eqref{doty=} (resp. Equation~\eqref{dotq=}) with the control $u$ and the initial condition $y(0)=y^0$ (resp. $q(0)=q^0).$

\begin{proof}[Proof of Theorem~\ref{th:part_cont}]
We first define $T_0.$
From Corollary~\ref{co:asymptotic} and Lemma~\ref{lm:orderYpi}, we know that
\[\forall y^0\in[0,Y(\bar u)),\ \exists!\, T_0(y^0)>0,\quad  (y,u_\text{max},y^0)\ \text{is solution to \eqref{doty=}}\implies y(T_0(y^0))=Y(\bar u),\]
and
\[\forall y^0\in(Y(\bar u),1],\ \exists!\, T_0(y^0)>0,\quad  (y,u_\text{min},y^0)\ \text{is solution to \eqref{doty=}}\implies y(T_0(y^0))=Y(\bar u).\]
The function $T_0$ thus defined is bounded on $[0,1].$
This is a consequence of the Cauchy-Lipschitz theorem for Equation~\eqref{doty=}, which ensures that for all $y\in[0,1],$ $T_0(y)\leq\max\{T_0(0),T_0(1)\}.$

The time $T_\psi$ is also defined by using Corollary~\ref{co:asymptotic} and Lemma~\ref{lm:orderYpi}
\beq\label{def:Tpsi}\exists!\, T_\psi>0,\quad  (q,u_\text{min},1/3)\ \text{is solution to \eqref{dotq=}}\implies q(-T_\psi)=\pi(\bar u).\eeq

Once $T_0$ and $T_\psi$ are defined, it only remains to verify that the control defined by \eqref{eq:upart} and the associated trajectories $y^*$ and $q^*$ satisfy the maximality condition~\eqref{eq:maximality2}.
 For this, it suffices to check that property \eqref{uaf} of the switching function $\Phi^*(t)=\Phi(x^*(t),p^*(t))$ holds (see Figure~\ref{fig:relax} for a numerical illustration).

In the interval $[T_0(y(0)),T-T_\psi]$, we have $u^*=\bar u$, $y^*\equiv Y(\bar u)$ and $q^*\equiv \pi(\bar u)$.
Thus, by Theorem~\ref{th:singular} and Theorem~\ref{th:perrondim2}, we have $\Phi^*=0$.

We have from the definition of $T_0$ that, if $y(0)>Y(\bar u),$ then $u^*\equiv u_{\min}$ and $y^*>Y(\bar u)$ in the interval $[0,T_0(y(0)))$.
Using Lemma~\ref{lm:Phidot=} we deduce from $y^*>Y(\bar u)$ that $\Lambda_2>0$ and then, since $\Phi(T_0(y(0)))=0$, that $\Phi^*<0$ in the interval $[0,T_0(y(0)))$.
The same argument proves that, if $y(0)<Y(\bar u)$, then $u^*\equiv u_{\max}$ and $\Phi^*>0$ in the interval $[0,T_0(y(0)))$.

In the interval $(T-T_\psi,T]$ we have $u^*(t)\equiv u_{\min},$ so we get from $y^*(T-T_\psi)=Y(\bar u),$ Corollary~\ref{co:asymptotic} and Lemma~\ref{lm:orderYpi} that $y^*(t)<Y(\bar u)$.
It follows from Lemma~\ref{lm:Phidot=} that for every $t\in(T-T_\psi,T],$ $\Lambda_2(t)<0$, and then $\Phi^*(t)<0$ since $\Phi^*(T-T_\psi)=0$.

\end{proof}

\

\begin{proof}[Proof of Theorem~\ref{th:uniqueness}]
Let $t\in [0,T]\mapsto u^*(t)$ be a control such that the maximality condition~\eqref{eq:maximality2} is satisfied. Let $x^*$, $p^*$, $y^*$, $q^*$ and $\Phi^*$ be the corresponding functions.
Since $\Phi^*(T)=\theta\tau x_1^*(T)<0$, necessarily, $u^*= u_{\min}$ in a neighborhood of $T$.
We start from time $T$ and analyze the problem backward in time.

\

{\it First step:}
we prove by contradiction that
\[\exists\eta,T_1>0\ \text{(both independent of $x^0$) s.t.}\quad T>T_1 \implies y^*(T)\not\in[Y(u_{\min})-\eta,Y(u_{\min})+\eta].\]
We start from the fact that
\begin{gather}
\label{phium}
\phi(u_{\min})(F+\theta G)X(u_{\min})=\lb_P'(u_{\min})>0.
\end{gather}
(The first equality of \eqref{phium} can be obtained by differentiating $(u(F+\theta G)+\zeta G) X(u)=\lambda_P(u)X(u)$ with respect to $u$ and by using $\phi(u)(u(F+\theta G)+\zeta G)=\lambda_P(u)\phi(u)$.) Property \eqref{phium} ensures that there exists $\varepsilon>0$ such that
\begin{gather}
\label{implimin}
\max\{|y-Y(u_{\min})|,|q-\pi(u_{\min})|\}\leq\varepsilon\quad\Longrightarrow\quad \Phi(x,p)>0.
\end{gather}
Using Lemma~\ref{lm:orderYpi} and decreasing $\varepsilon>0$ if necessary, we may assume that $\pi(u_{\min})-\varepsilon > 1/3$.
Then from Corollary~\ref{co:asymptotic} and Lemma~\ref{lm:orderYpi} we have that
\[\exists!\, T_{1,1}>0,\quad  (q,u_\text{min},1/3)\ \text{is solution to \eqref{dotq=}}\implies q(-T_{1,1})=\pi(u_{\min})-\varepsilon.\]
Consider now $\eta>0$ such that, for $(y,u_{\min},y^0)$ solution to \eqref{doty=}, \beq\label{def:eta}y^0\in[Y(u_{\min})-\eta, Y(u_{\min})+\eta]\implies\forall t\in[-T_{1,1},0],\ y(t)\in[Y(u_{\min})-\varepsilon,Y(u_{\min})+\varepsilon].\eeq
(This $\eta >0$ exists since $(y,u):=(Y(u_{\min}),u_{\min})$ is a solution of \eqref{doty=}.)
Still decreasing $\varepsilon>0$ if necessary and using Lemma~\ref{lm:orderYpi}, we may assume that
\begin{gather}
\label{Yuminpetit}
Y(u_{\min})+\varepsilon< Y(\bar u).
\end{gather}
By Corollary~\ref{co:defT}, Lemma~\ref{lm:orderYpi} and \eqref{Yuminpetit}
\[\exists!\, T_{1,2}>0,\quad  (y,u_{\max},Y(u_{\min})+\varepsilon)\ \text{is solution to \eqref{doty=}}\implies y(-T_{1,2})=0.\]
Let $T_1:=T_{1,1}+T_{1,2}$ and $T>T_1$.
Let us assume that $y^*(T)$ belongs to $[Y(u_{\min})-\eta, Y(u_{\min})+\eta]$.
If $\Phi^*<0$ in $[T-T_{1,1},T]$, then $u^*\equiv u_{\min}$ in $[T-T_{1,1},T]$, and, by the definition~\eqref{def:eta} of $\eta$, $y^*(T-T_{1,1})\in [Y(u_{\min})-\varepsilon, Y(u_{\min})+\varepsilon]$, which, together with \eqref{implimin} gives $\Phi^*(T-T_{1,1})>0$.
Hence $\Phi^*<0$ on $[T-T_{1,1},T]$ does not hold and there exists
$t_0\in [T-T_{1,1},T]$ such that $\Phi^*(t_0)=0$ and $\Phi^*<0$ in $(t_0,T]$.
From~\eqref{def:eta} we have $y^*(t_0)\leqslant Y(u_{\min})+\varepsilon< Y(\bar u)$, which, with Corollary~\ref{co:signPhi}, gives that $\Phi^*(t)>0$  for every $t\in [0,t_0).$
In particular, $u^*\equiv u_{\max}$ on $[0,t_0]$.
Finally, using the definition of $T_{1,2}>0$, we get that $y^*(T-t_0-T_{1,2})\leqslant 0$, which is not possible since $x^*>0$ in $[0,T]$.
So $y^*(T)$ cannot belong to $[Y(u_{\min})-\eta,Y(u_{\min})+\eta]$.

\

{\it Second step:} we prove that
\[\exists T_2>0,\qquad T>T_2 \implies y^*(T)\not\in[0,Y(u_{\min})-\eta].\]
 By Corollary~\ref{co:defT},
\[\exists!\,T_{2,1}>0,\quad (y,u_{\min},Y(u_{\min})-\eta)\ \text{is solution to \eqref{doty=}}\implies  y(-T_{2,1})=0\]
and, using additionally Lemma~\ref{lm:orderYpi},
\[\exists!\,T_{2,2}>0,\quad (y,u_{\max},Y(u_{\min})-\eta)\ \text{is solution to \eqref{doty=}}\implies  y(-T_{2,2})=0.\]
For $T>T_2:=T_{2,1}+T_{2,2}$, the terminal value $y^*(T)$ cannot belong to $[0,Y(u_{\min})-\eta]$.
If $y^*(T)$ is in $[0,Y(u_{\min})-\eta]$, then either $\Phi^*(t)<0$ on $[T-T_{2,1},T]$ and in this case $y^*(T-T_{2,1})\leq0,$ or there exists $t_0\in[T-T_{2,1},T]$ such that $\Phi^*(t_0)=0$ and in this case we have from Corollary~\ref{co:signPhi} that $\Phi^*(t)>0$ for $t<t_0$, so $u=u_{\max}$ and $y^*(t_0-T_{2,2})\leq0.$
Neither of these two cases is possible since $x^*>0$ on $[0,T]$.

\

{\it Third step:} we prove that
\[\exists T_3>0,\qquad T>T_3\implies y^*(T)\not\in[Y(u_{\min})+\eta,1]\setminus\{Y_\psi\}\]
where $Y_\psi:=y(T_\psi)$ for $(y,u_{\min},Y(\bar u))$ solution to \eqref{doty=}, with $T_\psi$ defined in~\eqref{def:Tpsi}.

We start from the fact that, by Corollary~\ref{co:defT},
\[\exists!\,T_{3,1}>0,\qquad (y,u_{\min},Y(u_{\min})+\eta)\ \text{is  solution of \eqref{doty=}}\implies y(-T_{3,1})=1.\]
 Let us assume that $T>T_{3,1}$. If $y^*(T)\in(Y(u_{\min})+\eta,1]$, then there exists a time $t_0>T-T_{3,1}$ such that $\Phi^*(t_0)=0$ because $y^*(t)$ has to be less than $1$ for every time in $[0,T]$. Increasing $t_0$ if necessary, we may also impose that $\Phi^*<0$ in $(t_0,T]$.
It is not possible to have $y^*(t_0)>Y(\bar u)$ because in this case, by Lemma~\ref{lm:Phidot=}, we would have $\Phi^*(t_0)=0$ with $(\Phi^*)'(t_0)=\Lambda_2(t_0)>0$, which cannot hold since $\Phi^*<0$ in $(t_0,T]$.
So we necessarily have $y^*(t_0)\leq Y(\bar u)$.

If $y^*(t_0)<Y(\bar u)$, then, by Corollary~\ref{co:signPhi}, $\Phi^*$ is positive on $[0,t_0).$
But by Corollary~\ref{co:defT} and Lemma~\ref{lm:orderYpi}
\[\exists!\,T_{3,2}>0,\qquad (y,u_{\max},Y(\bar u))\ \text{is  solution of \eqref{doty=}}\implies y(-T_{3,2})=0.\]
So we get that $y^*(t_0)$ cannot be less than $Y(\bar u)$ if $T>T_3:=T_{3,1}+T_{3,2}$.

Hence, for $T>T_3$, we have $y^*(t_0)=Y(\bar u)$.
We deduce from this identity, together with $\Phi^*(t_0)=0$, that $q^*(t_0)=\pi(\bar u)$ and it implies that $t_0=T-T_\psi$, where $T_\psi$ is defined in the proof of Lemma \ref{th:part_cont}.
As a consequence, the only possible value for $y^*(T)$ is $Y_\psi=y(T_\psi)\in (Y(\bar u),1)$, where $y$ is the solution of \eqref{doty=} with $u\equiv u_{\min}$ such that $y(0)=Y(\bar u)$.

\

Here, we have proved that, for $T>\max\{T_1,T_2,T_3\}$, the only possible control which satisfies the PMP takes the value $u_{\min}$ in $[T-T_\psi,T]$.
Moreover, the associated trajectories satisfy $y^*(T-T_\psi)=Y(\bar u)$ and $q^*(T-T_\psi)=\pi(\bar u)$.
Then, using Lemma~\ref{lm:signdotyq} and Corollary~\ref{co:signPhi} with the same kind of arguments as above, it is straightforward to check that the control  defined by~\eqref{eq:upart} is the only control which satisfies the PMP for $T>T_c:=\max\{T_1,T_2,T_3\}$.
Since any optimal control satisfies the maximality condition~\eqref{eq:maximality2}, we conclude that this is the only optimal control.

\end{proof}

\begin{corollary}
Asymptotically in $T$ we have the convergence
\[\lim_{T\to+\infty}\f{\ln(J(u^*))}{T}=\lb_P(\bar u).\]
\end{corollary}

A similar ergodic result is proved in~\cite{CalvezGabriel} in the case of dimension $n=3$ but without proving that the limit is $\lb_P(\bar u)$.

\

\section{Conclusion and perspectives}

We have modeled the PMCA protocol by a system of differential equations with a control term.
The analysis of the optimal control problem, which aims to maximize the efficiency of the PMCA, makes appear that the solution may not be a classical control but a relaxed one.
Such a theoretical optimal control cannot be realized experimentally.
Nevertheless it can be approached by altering incubation and sonication phases.

Our main result provides, in dimension 2, the optimal ratio between the two phases.
It is given by $R_\text{opt}:=\f{\Delta_{\max}}{\Delta_{\min}}=\frac{\bar u-u_{\min}}{u_{\max}-\bar u}$ (see~\eqref{eq:deltat}), where $\bar u\in(u_{\min},u_{\max})$ is the constant which maximizes the Perron eigenvalue in the convex hull of the original control set.
To approach the optimal relaxed control via an alternation of incubation and sonication phases, the switching frequency has to be high.
But the frequency is limited experimentally, for instance due to the warming engendered by the sonication.
To maintain the temperature of the sample at a reasonable level, sufficiently long rest phases  (corresponding to incubation) have to be intercalated between the sonication phases.
Considering such experimental constraints, a close-to-optimal strategy should be to switch as fast as possible between sonication and incubation phases, keeping the optimal ratio $R_\text{opt}$ between the respective durations.

Before practicing this strategy in a real PMCA protocol, the parameters of the model have to be estimated from experiments.
This requires precise measurements of the size distribution of the polymers and inverse problem methods as the one detailed in~\cite{DT} (see also the references therein).
To use our result in dimension 2, the sizes of the polymers have to be divided into two pools and the mean polymerization and fragmentation coefficients of the two pools have to be estimated.
If one wants to improve the accuracy of the method, a higher dimensional model should be used.
But it appears that, even for the case of three compartments, the mathematical analysis is much more delicate than in dimension two.

\

\section*{Appendix}

\appendix

\section{Perron eigenvalue}
\label{sec-Perron-eigenvalue}

For $u>0$ fixed and $r(u)>0,$ denote by $(\lb_P(u),X_P(u),\phi_P(u))$ the Perron eigenelements of the matrix $M(u):=uF+r(u)G$ defined by
\[\left\{\begin{array}{lll}
\lb_P X_P=M(u)X_P,&\quad X_P>0,&\quad\|X_P\|_1=1,\\
\lb_P \phi_P=\phi_PM(u),&\quad \phi_P>0,&\quad\phi_P X_P=1.
\end{array}\right.\]
In the current section we investigate the dependence of these elements on the control parameter $u.$
Such eigenelements also exist for the continuous growth-fragmentation equation~\eqref{eq:prionPDE} (see~\cite{DG} for instance),
and their dependence on parameters is investigated in~\cite{CDG,PG}.

\

The function $r:\R^+\to\R^{+*}$ is assumed to be continuous and bounded.
Theorem~\ref{th:Perron} is an immediate consequence of the following
Lemma~\ref{lm:positivity} and Theorem~\ref{th:expansion}.

\begin{lemma}\label{lm:positivity}The eigenelements $\lb_P$, $X_P$ and $\phi_P$ are continuous functions of $u$ on $\R^+$.
Moreover, we have
$$\lb_P(u)>0\ \text{for}\ u>0\qquad\text{and}\qquad\lim_{u\to0}\lb_P(u)=0.$$
\end{lemma}

\begin{proof}
Since the function $r$ is continuous, the coefficients of the matrix $M(u)$ depend continuously on $u$.
As a consequence, the characteristic polynomial of $M(u)$ varies continuously with $u$.
The first eigenvalue $\lb_P$ is the largest root of this characteristic polynomial and the Perron-Frobenius theorem ensures that the multiplicity of this root is 1.
So $\lb_P$ is a continuous function of $u$.

Let $u\geq0$ and $(u_k)_{k\in\N}$ be a positive sequence which converges to $u$.
Since $\|X_P(u_k)\|_1=1$ there exists a subsequence of $(X_P(u_k))_k$ which converges to a limit $X_\infty$.
By continuity of $\lb_P(u)$ and $M(u)$, this limit satisfies $M(u)X_\infty=\lb_P(u)X_\infty$ and $\|X_\infty\|_1=1$.
By uniqueness of the first eigenvector, we conclude that the whole sequence $(X_P(u_k))_k$ converges to $X_\infty=X_P(u)$ and so $X_P$ is a continuous function of $u$.
Since $(X_P(u_k))_k$ is a positive convergent sequence, it is lower bounded and we deduce from the normalization $\phi_PX_P=1$ that the sequence $(\phi_P(u_k))_k$ is bounded.
The same as for $X_P$, we conclude from the uniqueness of the adjoint eigenvector that $\phi_P$ is a continuous function of $u$.

Define $\Theta:=(1, 1,\cdots, 1)$ and, for $j=2,\cdots,n,$ $K_j:=\sum_{i=1}^{j-1}\kappa_{ij}>0.$
We have: $\Theta X_P=\|X_P\|_1=1$ since $X_P>0,$ $\Theta G=0$ from \eqref{defG}, and $\Theta F=(0,K_2\beta_2,\cdots,K_n\beta_n)^T$ from~\eqref{defF}.
So multiplying the identity $\lb_P X_P=M(u)X_P$ by $\Theta$ we get
$$\lb_P(u)=u\,\Theta FX_P\leq u\,\max_{2\leq j\leq n}K_j\beta_j$$
which ensures that $\lb_P(u)$ is positive for $u$ positive and tends to zero when $u\to0$.

\end{proof}

\begin{theorem}\label{th:expansion}
Under Assumptions~\eqref{as:r} and \eqref{as:tau}, we have the expansions
\begin{equation*}\begin{array}{l}
\dis k<l,\quad\implies\quad\lambda_P(u)=r_0\tau_1+\left[r_0^{k+1}
(\tau_{k+1}-(k+1)\tau_1)\prod_{i=1}^{k}
\frac{\tau_i}{\beta_{i+1}}\right]u^{-k}
+\underset{u\to +\infty}{o}(u^{-k}),\\
\dis k=l,\quad\implies\quad\lambda_P(u)=r_0\tau_1
+\left[r_0^{k+1}(\tau_{k+1}-(k+1)\tau_1)\prod_{i=1}^{k}
\frac{\tau_i}{\beta_{i+1}}+r_k\tau_1\right]u^{-k}+\underset{u\to +\infty}{o}(u^{-k}),
\vspace{.4cm}\\
\dis k>l,\quad\implies\quad\lambda_P(u)=r_0\tau_1+r_l
\tau_1u^{-l}+\underset{u\to +\infty}{o}(u^{-l}).
\end{array}\end{equation*}
\end{theorem}

This result can be related to Corollary~1 in \cite{CDG} which provides an expansion of the first eigenvalue for the continuous growth-fragmentation model.
The proof of Theorem~\ref{th:expansion} uses the following lemma which gives the asymptotic behavior of the eigenvector $X_P=(x_1,x_2,\cdots,x_n)^T$.

\begin{lemma}\label{lm:expeigenvector}
Assume that $r(u)$ admits a limit $r_0>0$ when $u$ tends to $+\infty$, then
\beq\label{eq:expeigenvector}\forall i\in[1,n],\qquad x_i(u)\underset{u\to +\infty}{\sim}r_0^{i-1}\prod_{j=1}^{i-1}\frac{\tau_j}{\beta_{j+1}}\,u^{1-i}.\eeq
\end{lemma}

\begin{proof}[Proof of Lemma~\ref{lm:expeigenvector}]
We prove by induction on $i$ that
\beq\label{eq:IH}\tag{IH}u^{i-1} x_i(u)\xrightarrow[u\to +\infty]{}r_0^{i-1}\prod_{j=1}^{i-1}\frac{\tau_j}{\beta_{j+1}}
\qquad\text{and}\quad u^{i-1} x_j(u)\xrightarrow[u\to +\infty]{}0, \qquad\forall j>i.\eeq
\underline{$i=1:$}
We have by definition
\beq\label{eq:eigeneqr}(r(u)G+u F)X_P(u)=\lambda_P(u) X_P(u)\qquad \text{with}\quad \|X_P(u)\|_1=1.\eeq
We use $\psi=(1,2,\cdots,n)$, which satisfies $\psi F=0$ (see \eqref{psiF}).
Testing~\eqref{eq:eigeneqr} against $\psi$ on the left, we obtain
\beq\label{eq:firstphitest}r(u)\psi G X_P(u)=\lambda_P(u)\psi X_P(u)\eeq
and so $\lambda_P(u)$ is bounded since $\|X_P(u)\|_1=1$ and $r$ is bounded.
Dividing by $u$ in \eqref{eq:eigeneqr}, we get
\beq\label{eq:divbyu}\left(\f{r(u)}{u}G+F\right)X_P(u)=\f{\lambda_P(u)}{u}X_P(u).\eeq
The sequence $X_P(u)$ is bounded and thus convergence occurs when $u\to +\infty$ for a subsequence.
But from~\eqref{eq:divbyu} the limit $X_P^\infty$ must satisfy $FX_P^\infty=0$ so the whole sequence converges to
$$X_P^\infty=\delta:=(1,0\,\cdots,0)^T.$$

\

\noindent\underline{$i\to i+1:$} ($i+1\leq n$) We have
$$u^{i}FX_P(u)=u^{i-1}\lambda_P(u) X_P(u)-u^{i-1}r(u)GX_P(u).$$
We consider the $n-i$ last lines of this matrix identity and find
\begin{align*}
\left(\begin{array}{ccc}
-\beta_{i+1}& & \\
 & &(2\kappa_{kj}\beta_j)\\
 &\ddots& \\
\text{\Large$0$}& & \\
 & &-\beta_n
\end{array}\right)
\left(\begin{array}{c}
u^{i}x_{i+1}(u)\\
\\
\vdots \\
\\
u^{i}x_n(u)
\end{array}\right)
&=
\left(\begin{array}{c}
(\lambda_P(u)+r(u)\tau_{i+1})u^{i-1}x_{i+1}(u)-r(u)\tau_{i}u^{i-1}x_{i}(u) \\
\\
\vdots \\
\\
\lambda_P(u)u^{i-1}x_n(u)-r(u)\tau_{n-1}u^{i-1}x_{n-1}(u)
\end{array}\right)\\
&\xrightarrow[u\to +\infty]{\text{by~\eqref{eq:IH}}}
\left(\begin{array}{c}
\dis-r_0\tau_{i}\cdot r_0^{i-1}\prod_{j=1}^{i-1}\frac{\tau_j}{\beta_{j+1}} \\
\\
\text{\Large$0$} \\
\\
\end{array}\right),
\end{align*}
which concludes the proof of Lemma~\ref{lm:expeigenvector}.
\end{proof}

\begin{proof}[Proof of Theorem~\ref{th:expansion}]
Notice that $k<n$ since $\tau_n=0$ and $\tau_1>0$.
Using~\eqref{eq:firstphitest}, the convergence of $X_P$ to $\delta=(1,0\,\cdots,0)^T$ and the convergence of $r$ to $r_0,$ we obtain that $\lambda_P(u)$ converges when $u\to +\infty$ and that the limit $\lb_P^\infty$ satisfies
\beq\label{eq:order0}r_0\psi G \delta=\lb_P^\infty \psi \delta,\eeq
which gives
\begin{equation}\label{lambdaPinfty=}
\lb_P^\infty=r_0\tau_1.
\end{equation}
Once we have this limit, we need to estimate the difference $\lb_P(u)-\lb_P^\infty$ when $u\to+\infty.$
To do so we make the difference between~\eqref{eq:firstphitest} and \eqref{eq:order0} which gives, by using~\eqref{lambdaPinfty=},
\begin{gather}\label{psiXp=}
(\lambda_P(u)-r_0\tau_1)\psi X_P(u)+r_0\tau_1\psi (X_P(u)-\delta)=r(u)\psi G(X_P(u)-\delta)+(r(u)-r_0)\psi G\delta.
\end{gather}
Now we use the result of Lemma~\ref{lm:expeigenvector} which gives an equivalent to $X_P(u)$ when $u\to+\infty,$ and Assumption~\eqref{as:r} which gives an equivalent to $r(u)-r_0$ when $u\to+\infty,$ to deduce an equivalent to $\lb_P(u)-\lb_P^\infty.$
Denoting $m:=\min{(k,l)}$, we obtain from~\eqref{psiXp=}
\begin{align*}
u^m(\lambda_P(u)-r_0\tau_1)\psi X_P(u)&=u^mr(u)\psi G(X_P(u)-\delta)+u^m(r(u)-r_0)\psi G\delta-u^mr_0\tau_1\psi (X_P(u)-\delta)\\
&=u^mr_0\psi(G-\tau_1Id)(X_P(u)-\delta)+u^m(r(u)-r_0)\psi GX_P(u)\\
&=u^mr_0\sum_{j=1}^{n-1}(\tau_j-j\tau_1)(x_j(u)-\delta_{1,j})+u^m(r(u)-r_0)\psi GX_P(u)\\
&=r_0\sum_{j=k+1}^{n-1}(\tau_j-j\tau_1)u^mx_j(u)+u^m(r(u)-r_0)\psi GX_P(u)\\
&\xrightarrow[u\to +\infty]{\eqref{eq:expeigenvector}}\1_{\{k\leq l\}}r_0^{k+1}(\tau_{k+1}-(k+1)\tau_1)\prod_{i=1}^{k}\frac{\tau_i}{\beta_{i+1}}+\1_{\{k\geq l\}}r_l\tau_1,
\end{align*}
where $\delta_{1,j}=1$ if $j=1$ and $\delta_{1,j}=0$ otherwise.
\end{proof}

\

\section{Floquet eigenvalue}
\label{secFloquet}

For a $T$-periodic control $u(t)$,
the Floquet theorem ensures that there is a Floquet eigenvalue $\lb_F[u(\cdot)]$ and a $T$-periodic function $X_F[u(\cdot)](t)$ solution to
$$\f{d}{dt}X_F(t)=[M(u(t))-\lb_F]X_F(t).$$
The Floquet eigenvalues can sometimes be compared to the Perron eigenvalues \cite{Lepoutre,PG2}.
Here we make periodic variations around the optimal constant control $u_\mathrm{opt}$ to find whether or not periodic controls can provide a better eigenvalue than $\lb_P(u_\mathrm{opt}).$

\

Consider directional perturbations $u(t) = u_\mathrm{opt} + \varepsilon \gamma(t)$, where $\gamma$ is a fixed $T$-periodic function and $\varepsilon$ a varying parameter.
For the sake of clarity, we denote by $\lb_F(\varepsilon) := \lambda_F[u_\mathrm{opt} + \varepsilon \gamma(\cdot)]$ the Floquet eigenvalue associated to $\varepsilon$,
$X_F(\varepsilon;t):=X_F[u_\mathrm{opt}+\varepsilon\gamma(\cdot)](t)$ the eigenfunction and $X_F'(\varepsilon;t):=\partial_\varepsilon X_F(\varepsilon;t)$ its derivative with respect to $\varepsilon$.
We also use the notation
$$\langle f\rangle:=\f1T\int_0^T f(t)\,dt$$
for the time average of any $T$-periodic function $f(t)$.

Now we compute the derivatives of $\lb_F(\varepsilon)$ which correspond to the directional derivatives of the Floquet eigenvalue at the point $u_\mathrm{opt}$.
This kind of differentiation technique is used in~\cite{M2} to prove the results about the optimization of the Perron eigenvalue in the case of the continuous cell division problem.
A formula which involves only the coefficients of the equation and the first eigenvectors is obtained for the first and second derivatives.
Here, the computation of the second derivative requires a basis of eigenvectors, and so cannot be extended to continuous models.
For $M(u_\mathrm{opt})$ diagonalizable, we choose two bases $(X_1,X_2,\cdots,X_n)$ and $(\phi_1,\phi_2,\cdots,\phi_n)$ of direct and adjoint eigenvectors associated to the eigenvalues $\lb_1=\lambda_P(u_\mathrm{opt})=\lb_F(0)\in \R$, $\lb_2\in \mathbb{C},\cdots ,\lb_n \in \mathbb{C}$
such that $\phi_i X_i=1$ and $\phi_i X_j=0$ if $i\neq j.$
Moreover, we choose $X_1$ positive and normalized to have $X_1=X_P(u_\mathrm{opt})=X_F(\varepsilon=0)$.

\begin{proposition}[First order condition]\label{prop:per:firstderiv}
We have
\[ \dfrac{d \lb_F}{d\varepsilon}(0) = \langle\gamma\rangle \dfrac{d\lambda_P}{du} (u_\mathrm{opt}) = 0\, . \]
Hence, $u_\mathrm{opt}$ is a critical point also in the class of periodic control.
\end{proposition}

As in~\cite{Bacaer2,Bacaer1}, the first derivative of the Floquet eigenvalue is zero and we need to go to the following order.

\begin{proposition}[Second order condition]\label{prop:per:secondderiv}
If $M(u_\mathrm{opt})$ is diagonalizable, we have
$$\dfrac{d^2 \lb_F}{d\varepsilon^2}(0) = \langle\gamma^2\rangle\phi_1M''(u_\mathrm{opt})X_1+2\sum_{i=2}^{n}\langle\gamma_i^2\rangle(\lambda_1-\lambda_i)(\phi_1 M'(u_\mathrm{opt}) X_i)(\phi_i M'(u_\mathrm{opt}) X_1),$$
where $\gamma_i(t):=\phi_i X_F'(0;t)(\phi_i M'(u_\mathrm{opt})X_1)^{-1}$ is the unique $T$-periodic solution to the ODE
\beq\label{eq:gammai}\dot{\gamma_i}(t)  + \lambda_1 \gamma_i(t) =  \gamma(t) + \lambda_i \gamma_i (t).\eeq
\end{proposition}

\begin{remark}
For $\gamma\equiv1$, we obtain the second derivative of the Perron eigenvalue
$$\dfrac{d^2\lambda_1}{du^2}(u_\mathrm{opt})=\phi_1M''(u_\mathrm{opt})X_1+2\sum_{i=2}^{n}\dfrac{(\phi_1 M' X_i)(\phi_i M' X_1)}{\lambda_1-\lambda_i},$$
which is negative since $u_\mathrm{opt}$ is a maximum.
This formula appears in \cite{terrytao}.
There exists a physical interpretation in terms of repulsive/attractive forces among the eigenvalues.
\end{remark}

\begin{proof}[Proof of Proposition~\ref{prop:per:firstderiv}]
First we give an expression of the first derivative for the Perron eigenvalue.
By definition, we have
$$M(u)X_P=\lambda_PX_P,$$
which provides by differentiation
\[\lambda_P'X_P+\lambda_PX_P'=M'(u)X_P+M(u)X_P'.\]
Testing against the adjoint eigenvector $\phi_P$, we obtain
$$\lambda_P'+\lambda_P\phi_PX_P'=\phi_PM'(u)X_P+\phi_PM(u)X_P'.$$
Since
\begin{gather}\label{phiPM=}
\phi_PM(u)=\lambda_P\phi_P,
\end{gather}
we have
$$\lambda_P'=\phi_PM'(u)X_P=\phi_P(r'(u)G+F)X_P.$$

Now, starting from the Floquet eigenvalue problem, we have
$$\partial_t X_F + \lb_F(\varepsilon) X_F = M(u_\mathrm{opt}+ \varepsilon \gamma) X_F\, ,$$
which provides by differentiation with respect to $\varepsilon$ that
\beq\label{eq:firstderiv_periodic} \partial_t X_F' + \lb_F'(\varepsilon) X_F + \lb_F(\varepsilon) X_F' =
\gamma(t) M'(u_\mathrm{opt}+\varepsilon\gamma) X_F  + M(u_\mathrm{opt}+\varepsilon\gamma) X_F'\, . \eeq
We test the preceding equation against $\phi_1$ and we evaluate at $\varepsilon = 0$. We obtain, using \eqref{phiPM=},
$$\partial_t(\phi_1 X_F')+\lb_F'=\gamma\phi_1M'(u_\mathrm{opt})X_1,$$
and, after integration in time,
$$\lb_F'=\left(\frac1T\int_0^T \gamma(t)\,dt\right)\phi_1M'(u_\mathrm{opt})X_1=\left(\frac1T\int_0^T \gamma(t)\, dt\right) \dfrac{d\lambda_P}{du} (u_\mathrm{opt}) = 0.$$
It proves the first order condition.
\end{proof}

\begin{proof}[Proof of Proposition~\ref{prop:per:secondderiv}]
We test \eqref{eq:firstderiv_periodic} against another adjoint eigenvector $\phi_i$ and we evaluate at $\varepsilon = 0$.
Using Proposition~\ref{prop:per:firstderiv} and denoting $\gamma_i(t) := \phi_i X_F'(t)(\phi_i M'(u_\mathrm{opt})X_1)^{-1}$, we obtain
\begin{gather}\label{dotgammi=}
\dot{\gamma_i}(t)  + \lambda_1 \gamma_i(t) =  \gamma(t) + \lb_i \gamma_i (t).
\end{gather}
Next, we differentiate \eqref{eq:firstderiv_periodic} with respect to $\varepsilon$ and we get
\begin{align*}
&\partial_t X_F'' + \lb_F''(\varepsilon) X_F + 2\lb_F'(\varepsilon) X_F'  + \lb_F X_F'' \\
=& \gamma^2 M''(u_\mathrm{opt}+\varepsilon \gamma)X_F + 2\gamma M'(u_\mathrm{opt}+\varepsilon\gamma) X_F'  + M(u_\mathrm{opt}+\varepsilon\gamma) X_F''\,.
\end{align*}
We evaluate at $\varepsilon = 0$ and we test against $\phi_1$ to find, using once more Proposition~\ref{prop:per:firstderiv},
\begin{align}
& \partial_t (\phi_1 X_F'') + \lb_F''(0) = \gamma^2 \phi_1M''(u_\mathrm{opt})X_1 + 2\gamma(t) \phi_1 M'(u_\mathrm{opt}) X_F'  .
\label{eq:extremal Floquet}
\end{align}
We decompose the unknown $X_F'$ along the basis $(X_1,\cdots,X_n)$
\[ X_F' = \sum_{i=1}^n \gamma_i(t) (\phi_i M'(u_\mathrm{opt})X_1) X_i  \, .\]
We have in particular
\[ \phi_1 M'(u_\mathrm{opt}) X_F' = \sum_{i=2}^{n}\gamma_i(t) (\phi_i M'(u_\mathrm{opt})X_1)(\phi_1 M'(u_\mathrm{opt}) X_i)\, .   \]
To conclude, we integrate \eqref{eq:extremal Floquet} on $[0,T]$ to obtain
\[\lb_F''(0) = \langle\gamma^2\rangle\phi_1 M''(u_\mathrm{opt})X_1 + 2 \sum_{i=2}^{n}\langle\gamma\gamma_i\rangle (\phi_i M'(u_\mathrm{opt})X_1)( \phi_1 M'(u_\mathrm{opt}) X_i)\]
and use the identity
\[ (\lb_1-\lb_i) \left(\int_0^T \gamma_i^2(t)\,dt \right) =  \int_0^T \gamma(t)\gamma_i(t)\,dt,   \]
which can be checked by multiplying \eqref{dotgammi=} by $\gamma_i$ and by integrating the result on $[0,T]$.
\end{proof}

Now we can prove Theorem~\ref{th:Floquet} stated in Section~\ref{sec:eigenvalue} by using the result of Proposition~\ref{prop:per:secondderiv}.

\begin{proof}[Proof of Theorem~\ref{th:Floquet}]
We consider the periodic function
$$\gamma(t)=\cos(\omega t)$$
and we denote by $\lb_F(\varepsilon,\omega)$ the Floquet eigenvalue corresponding to the periodic control $u_\mathrm{opt}+\varepsilon\cos(\omega t).$
We compute the $2\pi/\omega$-periodic solution $\gamma_i(t)$ to \eqref{eq:gammai}
$$\gamma_i(t)=\f{\lb_1-\lb_i}{\omega^2+(\lb_1-\lb_i)^2}\cos(\omega t)+\f{\omega}{\omega^2+(\lb_1-\lb_i)^2}\sin(\omega t)$$
and then we obtain the formula
$$\lb_F''(0,\omega)=\frac12\phi_1M''(u_\mathrm{opt})X_1+\sum_{i=2}^{n}\frac{\lambda_1-\lambda_i}{\omega^2+(\lambda_1-\lambda_i)^2}(\phi_1 M'(u_\text{opt}) X_i)(\phi_i M'(u_\text{opt}) X_1).$$
But we have
$$\phi_1 M''(u_\mathrm{opt})X_1=r''(u_\mathrm{opt})\phi_1 GX_1$$
and we can compute $\phi_1 GX_1$ by doing the appropriate combination of
$$\lambda_P(u_\mathrm{opt})=\phi_1 (r(u_\mathrm{opt})G+u_\mathrm{opt} F)X_1$$
and
$$\lambda_P'(u_\mathrm{opt})=0=\phi_1 (r'(u_\mathrm{opt})G+F)X_1$$
to obtain
$$\phi_1 M''(u_\mathrm{opt})X_1=\frac{r''(u_\mathrm{opt})}{r(u_\mathrm{opt})-u_\mathrm{opt} r'(u_\mathrm{opt})}\lambda_P(u_\mathrm{opt}).$$
Thus, we have the limit
$$\lim_{\omega\to +\infty}\lb_F''(0,\omega)=\frac12\frac{r''(u_\mathrm{opt})}{r(u_\mathrm{opt})-u_\mathrm{opt} r'(u_\mathrm{opt})}\lambda_P(u_\mathrm{opt}).$$
\end{proof}

\

\section{Alternative proof of Theorem~\ref{th:perrondim2}}\label{sec:altproof}

\begin{proof}[Second proof of Theorem~\ref{th:perrondim2}]

The characteristic polynomial of the matrix $uF+(\theta u+\zeta)G$ is
$$X^2+\bigl((\theta u+\zeta)\tau+u\beta\bigr)X-u(\theta u+\zeta)\beta\tau.$$
The discriminant of this polynomial is
\begin{align}\label{discriminant}
\Delta & = (\theta u+\zeta)^2\tau^2+u^2\beta^2+6u(\theta u+\zeta)\beta\tau \nonumber\\
& = (\theta^2\tau^2+\beta^2+6\theta\beta\tau)u^2+2\zeta\tau(\theta\tau+3\beta)u+\zeta^2\tau^2.
\end{align}
Since $0< u<\displaystyle\frac{\zeta}{-\theta}$, we have $\Delta>0$.
Define new relevant parameters
$$A := \theta\tau+\beta,\qquad B := \sqrt{-2\theta\beta\tau},\qquad C := \theta\tau+3\beta \qquad\text{and}\qquad D := \zeta\tau.$$
With these notations, discriminant \eqref{discriminant} writes
$$\Delta = (A^2-2B^2)u^2+2CD u+D^2,$$
and the first eigenvalue of $uF+(u\theta+\zeta)G$ is
$$\lambda_P(u) = \frac12\left(-Au-D+\sqrt{\Delta}\right)>0.$$
Differentiating twice this expression we get
\[\lb_P''(u)=\f{2\Delta\Delta''-(\Delta')^2}{4\Delta}\]
and the following lemma ensures, together with $\Delta>0,$ that this second derivative is negative.
\begin{lemma}\label{lm:lbsecond}For all $u$ we have
\[2\Delta\Delta''-(\Delta')^2<0.\]
\end{lemma}
\begin{proof}[Proof of Lemma~\ref{lm:lbsecond}.]
We compute $2\Delta\Delta''-(\Delta')^2=4D^2(A^2-2B^2-C^2)=-32D^2\beta^2<0.$
\end{proof}
We obtain that $\lb_P$ is a strictly concave function of $u$ and thereby, since it vanishes at the ends of the interval $(0,\f{\zeta}{-\theta}),$ it admits a unique critical point $\bar u$ which is the maximum.
We conclude noticing that $u_\mathrm{sing}$ is a critical point of $\lb_P.$

\

We can also check the identity $\bar u=u_\mathrm{sing}$ by computation.
The optimal value $\bar u$ satisfies $\lb_P'(\bar u)=0$, with
$$\lb_P'(u)=\frac12\left(\frac{\Delta'}{2\sqrt{\Delta}}-A\right).$$
To obtain $\bar u$, we solve the equation
\begin{equation}\label{eq:squared}\left(\Delta'\right)^2 = 4A^2\Delta,\end{equation}
which writes
$$-2B^2(A^2-2B^2)u^2-4B^2CDu+D^2(C^2-A^2)=0.$$
The discriminant of this binomial is
$${\mathfrak D} = 8\,A^2B^2D^2(C^2+2B^2-A^2) = 64\,A^2B^2D^2\beta^2,$$
and the roots are
$$u^\pm = \frac{D}{B}\,\frac{BC\pm2\beta A}{2B^2-A^2}.$$
A solution to $\lb_P'(u)=0$ is a solution to \eqref{eq:squared} which satisfies $\displaystyle\frac{\Delta'}{A}>0$.
The computations give
$$\frac{\Delta'(u^-)}{A}=4\beta\frac{D}{B}>0\qquad\text{and}\qquad\frac{\Delta'(u^+)}{A}=-4\beta\frac{D}{B}<0,$$
so $\lb_P'(u)=0$ has a unique solution which is
$$\bar u = u^- = \frac{D}{B}\,\frac{BC-2\beta A}{2B^2-A^2}.$$
Finally we write, from \eqref{eq:using},
\begin{align*}
u_{\mathrm{sing}} & = \frac{D}{B}\,\frac{2\beta+B}{C-2A+2B} \\
& = \frac{D}{B}\,\frac{(2\beta+B)(C-2A-2B)}{(C-2A)^2-4B^2} \\
& = \frac{D}{B}\,\frac{2\beta A-BC}{A^2-2B^2} \\
& = u^-=\bar u.
\end{align*}

\end{proof}

\

\section*{Acknowledgments}
The understanding of the PMCA technique and its mathematical modeling would not have been possible without the many discussions with Natacha Lenuzza and Franck Mouthon. Thanks a lot to them for their patience and their illuminating explanations.\\
The results about the eigenvalue problems owe much to Vincent Calvez. The authors are very grateful to him for his useful suggestions and advice.

%
%

\bibliographystyle{plain}
\bibliography{Prion}

\end{document}